%% file: ms.tex
\documentclass[11pt]{amsart}

\usepackage{amsmath,amssymb,amsfonts,mathtools}
\usepackage{graphicx,microtype,placeins,flafter}
\usepackage[margin=1in]{geometry}
\usepackage[authoryear,round]{natbib}
\usepackage[hidelinks,hypertexnames=false]{hyperref}

\numberwithin{equation}{section}

\theoremstyle{plain}
\newtheorem{theorem}{Theorem}[section]
\newtheorem{proposition}[theorem]{Proposition}
\newtheorem{lemma}[theorem]{Lemma}
\newtheorem{corollary}[theorem]{Corollary}
\theoremstyle{remark}
\newtheorem{remark}[theorem]{Remark}

\newcommand{\C}{\mathbb C}
\newcommand{\D}{\mathbb D}
\newcommand{\R}{\mathbb R}
\newcommand{\Pm}{\mathcal P}
\newcommand{\Rm}{\mathcal R}
\newcommand{\capc}{\operatorname{cap}}
\newcommand{\dist}{\operatorname{dist}}
\newcommand{\Arg}{\operatorname{Arg}}
\newcommand{\e}{\mathrm e}
\newcommand{\ii}{\mathrm i}
\newcommand{\whC}{\widehat{\mathbb C}}

\hypersetup{pdftitle={Convergence of repeated-pole rational approximation of the exponential on infinite sectors},pdfauthor={Fei Xue and Tianqi Zhang}}
\title[Convergence of repeated-pole exponential approximation]{Convergence of repeated-pole rational approximation of the exponential on infinite sectors}

\author{Fei Xue}
\address{School of Mathematical and Statistical Sciences, Clemson University,
Clemson, South Carolina 29634, USA}
\email{fxue@clemson.edu}
\thanks{Fei Xue was supported in part by the National Science Foundation under
grant DMS-2608905.}

\author{Tianqi Zhang}
\address{School of Data Sciences, Zhejiang University of Finance and Economics,
18 Xueyuan Street, Hangzhou 310018, China}
\email{tianqiz@zufe.edu.cn}

\begin{document}
\raggedbottom

\begin{abstract}
For $h>0$, degree $m\ge0$, $0\le\beta<\pi$, and $q>0$, let
$\Sigma_\beta=\{re^{i\theta}:r\ge0,\ |\theta|\le\beta/2\}$ and let
$E_{m,\beta}(q)$ denote the best uniform error for approximating $e^{-hz}$ on
$\Sigma_\beta$ by rational functions whose only possible finite pole is
$s=-q/h$, with multiplicity at most $m$. For the linear scaling
$q=\alpha m$ with a fixed $\alpha>0$, a M\"obius map reduces the problem to
polynomial approximation on a compact lens. Write
$\gamma_\beta=2-\beta/\pi$ and let $\alpha_c(\gamma_\beta)$ denote the
saddle-coalescence threshold. For
$0<\alpha<\alpha_c(\gamma_\beta)$, let
$H_\beta(\alpha)\in(0,1)$ be the exponential of the real saddle action. An
exact boundary Faber integral, a global fractional-power contour deformation,
and a two-saddle expansion yield the sharp two-sided order
$E_{m,\beta}(\alpha m)\asymp m^{-1/2}H_\beta(\alpha)^m$, uniformly on
compact subsets of $0<\alpha<\alpha_c(\gamma_\beta)$. On this
noncoalescing interval, the exponential factor $H_\beta(\alpha)$ has the
unique minimizer
$\alpha_\beta^*=1/[\gamma_\beta\sin(\pi/(2\gamma_\beta))]$ and minimum
$\tan(\pi/(4\gamma_\beta))$, recovering the classical half-line factor
$\tan(\pi/8)=\sqrt{2}-1$ at $\beta=0$. Numerical-range estimates give
matrix dimension-independent upper bounds for projected shift-and-invert
Arnoldi, while finite-degree pole searches and nonnormal experiments
illustrate the distinction between the sector-uniform law and individual
matrix computations.
\end{abstract}


\maketitle

\bigskip
\section{Introduction}\label{sec:introduction}

We study a restricted rational approximation problem for the matrix exponential in which all finite poles are concentrated at the same real point.  Its main computational advantage is that repeated use of the same shifted linear solve allows a single factorization or preconditioner to be reused.  This idea appears in restricted-denominator (RD) and rational Krylov methods for the exponential; see, for example, \citep{NorsettWolfbrandt1977,Norsett1978,MoretNovati2004,Novati2011, RostamiXue2018,GuettelShao2026a,GuettelShao2026b}. For broader polynomial and rational Krylov background, see also \citep{Saad1992,EshofHochbruck2006,Guettel2013}.

The scalar approximation problem considered here is posed on an \emph{infinite sector}, not on a bounded circular sector or a local corner neighborhood. For $h>0$ and $0\le\beta<\pi$, let $\Sigma_\beta=\{re^{\ii\theta}:r\ge0,\ |\theta|\le\beta/2\}$, where $\beta$ is the total opening. For an integer degree $m\ge0$ and $q>0$, let $E_{m,\beta}(q)$ denote the best uniform error for approximating $\e^{-hz}$ on $\Sigma_\beta$ by rational functions finite at infinity whose only possible finite pole is $s=-q/h$, with multiplicity at most $m$. Our asymptotic regime is $q=\alpha m$, where $\alpha>0$ is the fixed dimensionless pole-to-degree ratio.

Geometric convergence with one repeated pole on unbounded sets is classical. \citet{SaffSchonhageVarga1976} constructed an explicit family with denominator $(1+z/n)^n$ and proved geometric convergence on the half-line and on an infinite sector. \citet{SaffVargaNi1976} studied further Pad\'e and non-Pad\'e families in infinite sectors; related repeated-denominator formulas were investigated by \citet{VanIseghem1984,NorsettWolfbrandt1977,Norsett1978}. An early complex-domain note by \citet{KrestininPavlov1981} also considered one-pole approximation of the exponential in an unbounded region. The specific question pursued here is matching upper and lower asymptotics of the \emph{best} error as a function of the sector opening and a prescribed linear pole ratio.

On the half-line, the dependence on the prescribed pole ratio is known. \citet{Andersson1981} described the root-asymptotic error through a cubic saddle problem and found the optimal ratio $1/\sqrt2$ with root factor $\sqrt2-1$. Other real-pole approximation results include \citep{Lau1977,KaufmanTaylor1978,Borwein1983}; unrestricted rational minimax approximation, where the poles are free in the complex plane, is a different problem~\citep{CodyMeinardusVarga1969,Aptekarev2002}. The companion half-line paper~\citep{XueZhang2026} treats the uniform transition from fixed to linearly moving poles, including vanishing pole-to-degree ratios. It also resolves weighted $L^2$ oscillatory profiles and the mismatch between a nominal degree and the actual stopping degree. In particular, the $\beta=0$, $q=\alpha m$ specialization of the prefactor-resolved two-sided order proved below coincides with the exact-linear specialization of the companion result; we retain that endpoint as part of the sector family and as a normalization check. For genuine sectors $\beta>0$, the new obstacle is global: the Faber phase contains a noninteger power, and descending contours can encounter its negative-real branch cut. A boundary barrier and cancellation of the connecting segments identify the relevant saddle pair. Thus the present sector analysis shares the companion's M\"obius--saddle--noncancellation strategy but requires new fractional-power contour topology; it yields an angle-dependent best-error order, an explicit minimizer of the proved subcritical rate, and numerical-range consequences for nonnormal matrices. The lens map and the matrix transfer mechanism themselves are established tools.

The natural tool is Faber approximation. The M\"obius map used below sends the infinite sector to a compact spindle-shaped lens and converts the repeated-pole rational class exactly into polynomials. This geometry and the RD--Arnoldi/Faber framework are established tools \citep{MoretNovati2001,MoretNovati2004,Novati2003,Novati2011}; see also \citep{BeckermannReichel2009,GeorgievaHofreither2025}. The closest algorithmic comparisons are \citet{MoretNovati2004} and \citet{Novati2011}; their shift prescriptions and error estimates help identify the remaining issues about best approximation on the infinite sector.

In \citet[Cor.~3.4]{MoretNovati2004}, the Arnoldi error is bounded by an explicit factor \(C_\ell\) times an Arnoldi residual-polynomial factor. To compare parameters, denote their iteration count by $\ell$ (their rational degree is $\ell-1$), their scaled shift by $\tau$, and their sector half-angle by $\vartheta=\beta/2$. Their prescription $\tau=\ell$ minimizes the explicit factor $C_\ell$ with respect to $\tau$ and corresponds asymptotically to $q/m\to1$ in our notation. However, the remaining Arnoldi factor also depends on the shift and on the particular matrix--vector pair, so this optimization does not optimize the full bound or the scalar best error. More explicitly, put $s_\vartheta=\sin((\pi-2\vartheta)/4)$. Their formula~(3.23) and Stirling's formula give
\begin{equation*}
 C_\ell(\ell,\vartheta)
 =\frac{\e^\ell(\ell-1)!}{\pi\ell^\ell s_\vartheta^{2\ell+1}}
 \sim\frac{\sqrt{2/\pi}}{s_\vartheta\sqrt\ell}
       \left(\frac{2}{1-\sin\vartheta}\right)^\ell.
\end{equation*}
Thus even this minimized factor grows exponentially, and a decaying error bound requires compensation from the Arnoldi factor. Their Section~3 discussion gives root-superlinear decay of that residual-polynomial factor for a fixed compact operator, but supplies neither a quantitative onset nor a uniform estimate for changing shifted inverses in the finite-dimensional regime. In particular, it does not furnish a sharp finite-dimensional, sector-uniform rate for the practically relevant range in which the Arnoldi degree is much smaller than the matrix dimension. This observation concerns what these estimates alone imply; it does not preclude rapid convergence for particular operators.

\citet[Sec.~5, Prop.~12]{Novati2011} removes this matrix-dependent residual-polynomial factor from the sectorial bound by combining the RD--Arnoldi error representation with a Faber-capacity estimate. In the exponential case, his prescription corresponds asymptotically to $q/m\to\sec(\beta/2)$, and his explicit geometric factor, denoted here by $\rho_{\rm N}(\beta/2)$, satisfies $1/2\le\rho_{\rm N}(\beta/2)<1$ for $\beta/2<0.48124\ldots$; in the self-adjoint case, $\rho_{\rm N}(0)=1/2$. This gives an explicit mesh-independent geometric upper bound, but not the sharp scalar repeated-pole minimax rate. On the half-line, for example, the sharp scalar factor at $q=m$ is $0.422469\ldots$, below Novati's $1/2$, while optimizing the scalar ratio gives the classical value $\sqrt2-1$. The present work determines, for every fixed \(0\le\beta<\pi\), the best scalar error for each fixed ratio in the subcritical range, including both its sharp exponential factor and its algebraic scale, with matching upper and lower bounds, and then optimizes this exponential factor over the proved noncoalescing interval. The resulting matrix bounds are consequences of the scalar theorem and do not assert matching errors or optimal shifts for each matrix. Novati's a posteriori Arnoldi estimates and his treatment of other functions arising in exponential integrators concern complementary problems not addressed here.

The distinction from both earlier analyses is structural: the fixed lens capacity alone does not determine the diagonal rate because the target changes with the degree. The saddle action of the degree-dependent transformed exponential must also be included, and for a genuine sector its fractional power creates the global branch-cut and contour-topology problem treated below.

For each fixed $\beta$, write $\gamma=\gamma_\beta=2-\beta/\pi\in(1,2]$. The analysis proceeds in three steps. First, the lens map gives an exact boundary Faber coefficient integral with phase $\psi_{\alpha,\gamma}(\zeta)=\alpha\zeta^\gamma+\log((1-\zeta)/(1+\zeta))$ and saddle equation
\begin{equation}\label{eq:intro-saddle}
  \alpha\gamma_\beta\zeta^{\gamma_\beta-1}(1-\zeta^2)=2.
\end{equation}
We identify the relevant descending contours and obtain a two-saddle expansion (Theorem~\ref{thm:two-saddle}), uniform for $\alpha$ in compact subsets of the noncoalescing range $0<\alpha<\alpha_c(\gamma_\beta)$, where $\alpha_c(\gamma_\beta)$ is the saddle-coalescence threshold identified later. Second, two neighboring Faber coefficients cannot both be suppressed by destructive interference. Let $H_\beta(\alpha)$ denote the exponential of the real action at the relevant saddle. Together with tail estimates, Theorem~\ref{thm:best-error} gives the two-sided law
\begin{equation}\label{eq:intro-strong-law}
  E_{m,\beta}(\alpha m)\asymp
  m^{-1/2}H_\beta(\alpha)^m.
\end{equation}
Here $\asymp$ means bounds above and below by positive constants independent of $m$ and uniform on the chosen compact ratio set. 
Third, Proposition~\ref{prop:optimal-rate} optimizes the exponential factor \(H_\beta(\alpha)\) over \(0<\alpha<\alpha_c(\gamma_\beta)\). Its unique minimizer is \(\alpha_\beta^*\), with minimum value \(\rho_\beta^*\):
\begin{equation}\label{eq:intro-optimum}
  \alpha_\beta^*
  =\frac{1}{\gamma_\beta\sin\left(\frac{\pi}{2\gamma_\beta}\right)},
  \qquad
  \rho_\beta^*
  =\tan\!\left(\frac{\pi}{4\gamma_\beta}\right).
\end{equation}
At $\beta=0$ these reduce to the well-known half-line values of \cite{Andersson1981}. 

The curvature of \(\log H_\beta\) and the uniform two-sided law imply (Corollary~\ref{cor:finite-degree-minimizers}) that any minimizer \(\widehat\alpha_m\) of the finite-degree best error \(E_{m,\beta}(\alpha m)\) over a compact subcritical interval containing \(\alpha_\beta^*\) satisfies \(|\widehat\alpha_m-\alpha_\beta^*|=O(m^{-1/2})\). The root factors also quantify the degree overhead of the angle-independent rule $q=m$, providing a sharp best-approximation benchmark for this simple choice.

All asymptotic constants may depend on the fixed angle \(\beta\); no joint uniformity as $\beta\uparrow\pi$ is asserted. The scalar law has two consequences: (i) any polynomial interpolation scheme with subexponentially growing Lebesgue constants inherits the exponential factor \(H_\beta(\alpha)\) of the best-approximation error, and (ii) the uniform scalar error on the sector transfers to matrix functions through the numerical range. The repeated-pole class is exactly one shift-and-invert Krylov search space, so the subcritical rate-minimizing scalar pole gives one repeated real shift that can reuse a single matrix LU factorization or preconditioner.

The remainder of the paper is organized as follows. Section~\ref{sec:geometry} fixes the exact M\"obius and Faber geometry. Section~\ref{sec:coefficients} performs the saddle analysis. Section~\ref{sec:error-transfer} turns coefficients into best errors and optimizes the pole ratio. Section~\ref{sec:numerics} checks the scalar formulas and finite-degree pole selection; Section~\ref{sec:matrix} gives the matrix consequences and an Arnoldi experiment, and Section~\ref{sec:conclusion} concludes this study. For a first reading, the line of proof is: infinite sector $\longrightarrow$ compact lens $\longrightarrow$ boundary Faber integral $\longrightarrow$ two saddles $\longrightarrow$ two-sided best-error order $\longrightarrow$ the minimized exponential factor.

\section{M\"obius reduction and sectorial Faber geometry}\label{sec:geometry}

In this section, we transform the unbounded sector problem onto a fixed compact set.  The M\"obius transformation converts the repeated-pole rational class exactly into polynomials on a lens; the exterior map of that lens then supplies the Faber coordinate used in the rest of the paper.  We write $\D:=\{z\in\C:|z|<1\}$ for the unit disk and $\whC:=\C\cup\{\infty\}$ for the Riemann sphere.

Fix $h>0$ and $0\le \beta<\pi$, with $\beta$ fixed throughout the asymptotic analysis, and write $\Sigma_\beta:=\{re^{\ii\theta}:r\ge0,\ |\theta|\le\beta/2\}$.  We approximate $f_h(z)=\e^{-hz}$ on $\Sigma_\beta$, with $f_h(\infty):=0$, by rational functions whose only finite pole is a prescribed negative real point, repeated with multiplicity at most the approximation degree. For $q>0$, introduce the M\"obius transformation
\begin{equation}\label{eq:mobius}
  w=\frac{q-hz}{q+hz},
  \qquad
  hz=q\,\frac{1-w}{1+w}.
\end{equation}
The corresponding physical pole is $s=-q/h$, and the transformed target function is $g_q(w):=\exp\!\left(-q(1-w)/(1+w)\right)$.  The complex continuation of $g_q$ has an essential singularity at $w=-1$, but on the compact set introduced below it extends continuously to that point by $g_q(-1)=0$.

\subsection{Exact polynomial--rational correspondence}\label{subsec:isometry}

Define the compactified image
\begin{equation}\label{eq:Kbeta}
  K_\beta
  :=\left\{\frac{1-y}{1+y}:
       y=re^{\ii\theta},\ r\ge0,\ |\theta|\le\frac{\beta}{2}\right\}
     \cup\{-1\}.
\end{equation}
The added point $-1$ is the image of $y=\infty$.  Since a sector is invariant under multiplication by a positive real number, $K_\beta$ depends only on $\beta$, not on $q$ or $h$.

Write $\|f\|_X:=\sup_{z\in X}|f(z)|$ for the uniform norm on a set $X$. Let $\Pm_m$ denote the polynomials of degree at most $m$, and define the prescribed repeated-pole class
\begin{equation}\label{eq:Rmq}
  \Rm_m(q;h)
  :=\left\{p\!\left(\frac{q-hz}{q+hz}\right):p\in\Pm_m\right\}.
\end{equation}

\begin{proposition}[Exact M\"obius isometry]\label{prop:isometry}
For every $m\ge0$, $q>0$, $h>0$, the map $p(w)\mapsto r(z):=p((q-hz)/(q+hz))$ is a bijection from $\Pm_m$ onto $\Rm_m(q;h)$.  Every $r\in\Rm_m(q;h)$ has no finite pole except $z=-q/h$, with multiplicity at most $m$, and
\begin{equation}\label{eq:isometry}
  \sup_{z\in\Sigma_\beta}|\e^{-hz}-r(z)|
  =\sup_{w\in K_\beta}|g_q(w)-p(w)|.
\end{equation}
Consequently, if
\begin{equation}\label{eq:best-error}
  E_{m,\beta}(q)
  :=\inf_{p\in\Pm_m}\|g_q-p\|_{K_\beta},
\end{equation}
then $E_{m,\beta}(q)$ is exactly the best uniform error in the corresponding repeated-pole rational class on $\Sigma_\beta$.
\end{proposition}

\begin{proof}
The inverse of \eqref{eq:mobius} is $z=(q/h)(1-w)/(1+w)$. A polynomial of degree at most $m$ in $w$ therefore becomes a rational function whose denominator divides $(q+hz)^m$, and the converse follows by the same substitution.  The map \eqref{eq:mobius} is a bijection from the compactified sector $\Sigma_\beta$ onto $K_\beta$, including $z=\infty\leftrightarrow w=-1$.  The norm identity \eqref{eq:isometry} and then \eqref{eq:best-error} follow immediately.
\end{proof}

Thus $h$ disappears from the scalar problem and re-enters only through the physical pole $s=-q/h$.

\subsection{The spindle-shaped compact set}\label{subsec:spindle}

For $\beta>0$, each boundary ray of $\Sigma_\beta$ is mapped by \eqref{eq:mobius} to a circular arc joining $-1$ and $1$.  A direct calculation gives
\begin{equation}\label{eq:lens-disks}
  K_\beta
  =D\!\left(\ii\cot\frac{\beta}{2},\csc\frac{\beta}{2}\right)
   \cap
   D\!\left(-\ii\cot\frac{\beta}{2},\csc\frac{\beta}{2}\right),
\end{equation}
where $D(c,r)$ is the closed disk of center $c$ and radius $r$.  Hence $K_\beta$ is a convex lens, symmetric with respect to both coordinate axes; $K_0=[-1,1]$ and $K_\beta\to\overline{\D}$ as $\beta\uparrow\pi$.

\subsection{Exact exterior map and logarithmic capacity}\label{subsec:exterior-map}

Set $\gamma=\gamma_\beta:=2-\beta/\pi\in(1,2]$.  The complement of the original sector has opening angle \(2\pi-\beta=\pi\gamma\). Let $Y(w):=(1-w)/(1+w)$.  On $\whC\setminus K_\beta$, choose the branch with
\begin{equation}\label{eq:Y-arg}
  \frac{\beta}{2}<\Arg Y(w)<2\pi-\frac{\beta}{2},
\end{equation}
and define
\begin{equation}\label{eq:Ubeta}
  U_\beta(w)
  :=\e^{-\ii\pi/\gamma}Y(w)^{1/\gamma}.
\end{equation}
Then $U_\beta$ maps $\whC\setminus K_\beta$ conformally onto the right half-plane.  The Cayley transform therefore gives the normalized exterior map
\begin{equation}\label{eq:Phi}
  \Phi_\beta(w)
  :=\frac{1+U_\beta(w)}{1-U_\beta(w)}
\end{equation}
from $\whC\setminus K_\beta$ onto $\whC\setminus\overline\D$, with $\Phi_\beta(\infty)=\infty$.

\begin{proposition}[Exterior normalization and capacity]\label{prop:capacity}
As $w\to\infty$,
\begin{equation}\label{eq:Phi-expansion}
  \Phi_\beta(w)
  =\gamma w+\frac{1-\gamma^2}{3\gamma}\,w^{-1}
   +O(w^{-3}).
\end{equation}
Consequently,
\begin{equation}\label{eq:capacity}
  \capc(K_\beta)=\frac1\gamma=\frac{\pi}{2\pi-\beta}.
\end{equation}
\end{proposition}

\begin{proof}
Writing $x=w^{-1}$, the branch choice \eqref{eq:Y-arg} gives $U_\beta(w)=((1-x)/(1+x))^{1/\gamma}$. Expanding at $x=0$ and substituting into \eqref{eq:Phi} gives \eqref{eq:Phi-expansion}. For the exterior map normalized by $\Phi_\beta(\infty)=\infty$ and $\Phi_\beta'(\infty)>0$, the coefficient of $w$ is the reciprocal logarithmic capacity.  This proves \eqref{eq:capacity}.
\end{proof}

In particular, $\capc(K_0)=1/2$ and $\capc(K_\beta)\uparrow1$ as $\beta\uparrow\pi$.  The capacity describes the fixed lens but not the diagonal approximation rate.  This geometry is known: in \citet[Sec.~5]{Novati2011}, identify his angle parameter with $\theta=\beta/2$, set $\nu=2\theta/\pi$, and apply the affine change from his lens variable to ours.  Then $2-\nu=\gamma$, his inverse exterior map becomes \eqref{eq:Psi}, and the capacity scales from $1/[2(2-\nu)]$ to $1/\gamma$.  We include the derivation because the present normalization leads directly to \eqref{eq:g-Psi}.

\subsection{Inverse map and the Faber coordinate}\label{subsec:inverse-map}

Let $\xi=\Phi_\beta(w)$ and $\zeta:=(\xi-1)/(\xi+1)$.  For $|\xi|>1$, one has $\Re\zeta>0$.  Solving \eqref{eq:Phi} for $w$ gives the exterior inverse map
\begin{equation}\label{eq:Psi}
  \Psi_\beta(\xi):=\Phi_\beta^{-1}(\xi)
  =\frac{1+\zeta^\gamma}{1-\zeta^\gamma},
  \qquad
  \zeta=\frac{\xi-1}{\xi+1},
\end{equation}
where $\zeta^\gamma$ is taken on the branch $-\pi/2<\arg\zeta<\pi/2$; moreover $\Psi_\beta(\xi)=\xi/\gamma+O(\xi^{-1})$ as $\xi\to\infty$.  The identity
\begin{equation}\label{eq:key-Psi-identity}
  \frac{1-\Psi_\beta(\xi)}{1+\Psi_\beta(\xi)}=-\zeta^\gamma
\end{equation}
then yields the particularly simple exterior representation
\begin{equation}\label{eq:g-Psi}
  g_q(\Psi_\beta(\xi))=\e^{q\zeta^\gamma}.
\end{equation}

\begin{remark}[Recovery of the interval geometry]\label{rem:joukowski}
If $\beta=0$, then $\gamma=2$ and \eqref{eq:Psi} reduces to \( \Psi_0(\xi)=\frac12(\xi+\xi^{-1})\), the classical Joukowski map.  Thus the sectorial construction reduces exactly to the Chebyshev geometry on $[-1,1]$.
\end{remark}

\subsection{Faber polynomials}\label{subsec:faber-polys}

Let $F_n^{(\beta)}$ be the $n$th Faber polynomial associated with $K_\beta$, namely the polynomial part of the Laurent expansion of $\Phi_\beta(w)^n$ at infinity.  We suppress the superscript when the angle is fixed.  Since $K_\beta$ is convex, the standard Faber bound gives
\begin{equation}\label{eq:faber-bound}
  \|F_n\|_{K_\beta}\le2,\qquad n\ge1;
\end{equation}
see, for example, \citep[Ch.~I]{Suetin1998}.  For $\beta=0$, $F_n(w)=2T_n(w)$ for $n\ge1$, where $T_n$ is the Chebyshev polynomial of the first kind, so the present Faber basis reduces to the Chebyshev basis of the half-line approximation problem.

\section{Faber coefficients and fixed-ratio saddle asymptotics}\label{sec:coefficients}

We now come to the main asymptotic calculation.  The target varies with the degree, so the Faber coefficient index and the approximation degree must be kept separate.  We use $n$ for a coefficient and $m$ for the final approximation.  If $q=\alpha m$, then the omitted coefficient of index $m+j$ has effective ratio $q/(m+j)=\alpha m/(m+j)$, which differs from $\alpha$ by $O(m^{-1})$ for fixed $j$.

For the saddle analysis itself we therefore study the diagonal linear regime $q=\alpha n$, with $\alpha$ in a compact subinterval of the noncoalescing range; no uniform analysis as $\alpha\downarrow0$ is needed.

\subsection{An exact coefficient integral}\label{subsec:coeff-integral}

The starting point is an exact boundary coefficient, not an exterior level-curve coefficient.  This distinction matters because the transformed exponential has an essential singularity at the boundary point corresponding to infinity in the original sector.

For $n\ge1$, define the boundary Faber coefficient functional of $g_q$ by
\begin{equation}\label{eq:faber-coeff-xi}
  a_n(q)
  :=\frac{1}{2\pi\ii}\int_{|\xi|=1}
  g_q(\Psi_\beta(\xi))\,\xi^{-n-1}\,d\xi.
\end{equation}
The integral is understood as the boundary limit at $\xi=-1$.  It is absolutely convergent there because the real part of the exponent in \eqref{eq:g-Psi} tends to $-\infty$ along the unit circle.  For a polynomial $p$ of degree smaller than $n$, the corresponding functional vanishes, since $p(\Psi_\beta(\xi))$ has no $\xi^n$ term in its Laurent expansion at infinity.

\begin{remark}[Boundary rather than exterior coefficient]\label{rem:boundary-coeff}
The definition \eqref{eq:faber-coeff-xi} is a boundary Faber coefficient because $g_q$ is singular at $w=-1$ ($\xi=-1$).  Although $g_q\circ\Psi_\beta$ is analytic for $|\xi|>1$, its behavior near $\xi=-1$ is not uniformly controlled as $|\xi|\downarrow1$, so \eqref{eq:faber-coeff-xi} cannot be replaced by an exterior level-curve integral moved to infinity.  For polynomials the exterior Laurent expansion is regular to the required order, hence the functional annihilates $\Pm_{n-1}$.  This is the sectorial analogue of a Fourier--Chebyshev coefficient at a flat endpoint.
\end{remark}

\begin{lemma}[Exact imaginary-axis representation]\label{lem:imag-axis}
For every $q>0$ and integer $n\ge1$,
\begin{equation}\label{eq:coeff-exact}
  a_n(q)
  =\frac{1}{\pi\ii}
  \int_{-\ii\infty}^{\ii\infty}
  \e^{q\zeta^\gamma}
  \frac{(1-\zeta)^{n-1}}{(1+\zeta)^{n+1}}\,d\zeta.
\end{equation}
Equivalently, locally away from $\zeta=\pm1$,
\begin{equation}\label{eq:coeff-phase}
  a_n(q)
  =\frac{1}{\pi\ii}\int_{\ii\R}
  \frac{1}{1-\zeta^2}
  \exp\!\left\{
  q\zeta^\gamma+n\log\frac{1-\zeta}{1+\zeta}
  \right\}d\zeta.
\end{equation}
Here $\zeta^\gamma$ on the imaginary axis is the boundary value from the right half-plane.
\end{lemma}

\begin{proof}
On $|\xi|=1$, set $\zeta=(\xi-1)/(\xi+1)$, so $\xi=(1+\zeta)/(1-\zeta)$ and $d\xi=2(1-\zeta)^{-2}d\zeta$. The positively oriented unit circle maps onto $\ii\R$, from $-\ii\infty$ to $\ii\infty$. Formula \eqref{eq:g-Psi} gives $g_q(\Psi_\beta(\xi))=\e^{q\zeta^\gamma}$, while $\xi^{-n-1}d\xi=2(1-\zeta)^{n-1}(1+\zeta)^{-n-1}d\zeta$. This proves \eqref{eq:coeff-exact}; factoring the rational term gives \eqref{eq:coeff-phase}. Finally, since $\gamma=2-\beta/\pi$, $\Re((\ii y)^\gamma)=-|y|^\gamma\cos(\beta/2)$, so the exponential factor decays at both ends for every $0\le\beta<\pi$.
\end{proof}

\begin{remark}[Half-line normalization]\label{rem:halfline-coeff}
For $\beta=0$, $F_n=2T_n$, so $a_n$ is half the usual Chebyshev coefficient and \eqref{eq:coeff-exact} is the standard transformed half-line coefficient integral.
\end{remark}

\subsection{The diagonal phase and its saddles}\label{subsec:saddles}

When $q$ is proportional to the coefficient index, $n$ is the large parameter and all pole dependence sits in a fixed phase. The saddle equation generalizes Andersson's cubic.

Set $q=\alpha n$.  Then \eqref{eq:coeff-phase} takes the phase--amplitude form
\begin{equation}\label{eq:phase-form}
  a_n(\alpha n)
  =\frac{1}{\pi\ii}\int_{\ii\R}
  A(\zeta)\e^{n\psi_{\alpha,\gamma}(\zeta)}\,d\zeta,
\end{equation}
where
\begin{equation}
  A(\zeta)=\frac{1}{1-\zeta^2}, \quad\mbox{ and }\quad \label{eq:phase}
  \psi_{\alpha,\gamma}(\zeta)
  :=\alpha\zeta^\gamma
  +\log\frac{1-\zeta}{1+\zeta}.
\end{equation}
Near the first-quadrant saddle, $\zeta^\gamma$ and the logarithms in $\psi_{\alpha,\gamma}$ are taken on their principal branches, equivalently $\log((1-\zeta)/(1+\zeta))=\log(1-\zeta)-\log(1+\zeta)$ there; near the lower saddle we use the conjugate branch.  These choices agree with the upper- and lower-half-plane phases used in the global deformation below.  The logarithmic phase is used only locally away from $\zeta=\pm1$; global deformations use the exact integrand \eqref{eq:coeff-exact}.

The saddle equation is
\begin{equation}\label{eq:saddle-eq}
  \alpha\gamma\zeta^{\gamma-1}(1-\zeta^2)=2.
\end{equation}
For $\gamma=2$ this becomes $\alpha(\zeta^3-\zeta)+1=0$, exactly the cubic saddle equation of the half-line problem.

\subsection{Parametrization and the coalescence threshold}\label{subsec:parametrization}

A polar parametrization makes the saddle motion one-dimensional.  It also shows exactly where the conjugate pair coalesces and therefore where the ordinary Gaussian saddle expansion must stop.  The point of the calculation below is not merely to solve the saddle equation explicitly.  The angle $\theta$ will serve as a global parameter for the relevant saddle branch: $\theta\downarrow0$ corresponds to the collision of the conjugate pair on the positive real axis, while $\theta\uparrow\pi/(\gamma+1)$ drives the pole ratio toward zero.  Thus monotonicity of $\alpha_\gamma(\theta)$ simultaneously identifies the subcritical branch and its endpoint.

Write a first-quadrant saddle as $\zeta=r\e^{\ii\theta}$ with $r>0$ and $0<\theta<\pi/(\gamma+1)$.  The imaginary part of \eqref{eq:saddle-eq} gives
\begin{equation}\label{eq:rtheta}
  r^2=\frac{\sin((\gamma-1)\theta)}
             {\sin((\gamma+1)\theta)}.
\end{equation}
Moreover, $1-r^2\e^{2\ii\theta}=[\sin(2\theta)/\sin((\gamma+1)\theta)]\e^{-\ii(\gamma-1)\theta}$.  Taking moduli in \eqref{eq:saddle-eq} therefore yields
\begin{equation}\label{eq:alpha-theta}
  \alpha_\gamma(\theta)
  =\frac{2}{\gamma}
  \frac{\sin((\gamma+1)\theta)^{(\gamma+1)/2}}
       {\sin(2\theta)\,
        \sin((\gamma-1)\theta)^{(\gamma-1)/2}}.
\end{equation}

\begin{lemma}[Monotone saddle parametrization]\label{lem:alpha-monotone}
The function $\theta\mapsto\alpha_\gamma(\theta)$ is strictly decreasing on $0<\theta<\pi/(\gamma+1)$ and maps this interval bijectively onto $(0,\alpha_c(\gamma))$, where
\begin{equation}\label{eq:alpha-c}
  \alpha_c(\gamma)
  :=\frac1\gamma
  \frac{(\gamma+1)^{(\gamma+1)/2}}
       {(\gamma-1)^{(\gamma-1)/2}}.
\end{equation}
\end{lemma}

\begin{proof}
For $0<a<b$ and $0<\theta<\pi/b$, the ratio $\sin(a\theta)/\sin(b\theta)$ is strictly increasing: its logarithmic derivative is $a\cot(a\theta)-b\cot(b\theta)>0$ because $x\cot x$ decreases on $(0,\pi)$. Applying this first with $(a,b)=(\gamma-1,\gamma+1)$ shows that $r^2$ in \eqref{eq:rtheta} increases, and then with $(a,b)=(2,\gamma+1)$ shows that $\sin(2\theta)/\sin((\gamma+1)\theta)$ increases.  From the modulus form of the saddle equation, $\alpha_\gamma(\theta)=2\sin((\gamma+1)\theta)/[\gamma r(\theta)^{\gamma-1}\sin(2\theta)]$, so $\alpha_\gamma$ is strictly decreasing.

As $\theta\downarrow0$, \eqref{eq:alpha-theta} gives \eqref{eq:alpha-c}.  As $\theta\uparrow\pi/(\gamma+1)$, the numerator in \eqref{eq:alpha-theta} tends to zero while the remaining factors have the corresponding positive powers needed to give $\alpha_\gamma(\theta)\to0$.  Strict monotonicity proves the bijection.
\end{proof}

\begin{lemma}[Relevant saddle pair]\label{lem:saddle-pair}
For every $0<\alpha<\alpha_c(\gamma)$, equation \eqref{eq:saddle-eq} has a unique saddle $\zeta_{\alpha,\gamma}$ in the first quadrant and its conjugate $\overline{\zeta_{\alpha,\gamma}}$ in the fourth quadrant.  The first-quadrant saddle is characterized by \eqref{eq:rtheta}--\eqref{eq:alpha-theta}.  The pair coalesces on the positive real axis at $\alpha=\alpha_c(\gamma)$, at
\begin{equation}\label{eq:zeta-c}
  \zeta_c=\sqrt{\frac{\gamma-1}{\gamma+1}}.
\end{equation}
Furthermore,
\begin{equation}\label{eq:second-derivative}
  \psi''_{\alpha,\gamma}(\zeta_{\alpha,\gamma})
  =\frac{2\big[(\gamma-1)-(\gamma+1)\zeta_{\alpha,\gamma}^2\big]}
  {\zeta_{\alpha,\gamma}(1-\zeta_{\alpha,\gamma}^2)^2},
\end{equation}
so the saddle is nondegenerate throughout $0<\alpha<\alpha_c(\gamma)$.
\end{lemma}

\begin{proof}
There are two points to check.  The parametrization already gives one saddle in the first quadrant; what remains is to exclude any additional upper-half-plane solutions and then to identify the only possible loss of nondegeneracy. Lemma~\ref{lem:alpha-monotone} gives existence and uniqueness for $0<\theta<\pi/(\gamma+1)$.  No first-quadrant solution can occur for $\pi/(\gamma+1)\le\theta<\pi/2$, where the identity leading to \eqref{eq:rtheta} would require $r^2=\sin((\gamma-1)\theta)/\sin((\gamma+1)\theta)\le0$, whereas $r>0$.  At $\theta=\pi/2$, the left-hand side of \eqref{eq:saddle-eq} has argument $(\gamma-1)\pi/2\ne0$, so the imaginary axis contains no saddle. There is also no additional upper-half-plane solution with $\pi/2<\theta<\pi$.  Indeed, write $\zeta=r\e^{\ii\theta}$ in that range.  As $r$ increases from $0$ to infinity, the principal argument of $1-r^2\e^{2\ii\theta}$ increases from $0$ to $2\theta-\pi$.  Hence the argument of $\zeta^{\gamma-1}(1-\zeta^2)$ stays strictly between $0$ and $(\gamma+1)\theta-\pi<2\pi$ and can never be congruent to $0$ modulo $2\pi$.  Thus the upper saddle is unique, and conjugation gives the unique lower saddle.  At coalescence, the limiting angle is $\theta=0$, and \eqref{eq:rtheta} yields \eqref{eq:zeta-c}.  Finally,
\[
  \psi''_{\alpha,\gamma}(\zeta)
  =\alpha\gamma(\gamma-1)\zeta^{\gamma-2}
    -\frac{4\zeta}{(1-\zeta^2)^2}.
\]
Eliminating $\alpha$ by \eqref{eq:saddle-eq} gives \eqref{eq:second-derivative}.  Its numerator vanishes precisely when $\zeta^2=(\gamma-1)/(\gamma+1)$, i.e., at the coalescence point.
\end{proof}

For $\gamma=2$, \eqref{eq:alpha-c} gives $\alpha_c(2)=3\sqrt3/2$, in agreement with the half-line cubic.

\subsection{The sectorial rate and saddle phase}\label{subsec:rate-function}

The real part of the saddle action gives the geometric decay, while its imaginary part controls coefficient oscillation.  We name both quantities now because the first will become the best-error root factor and the second will be used later to prevent cancellation.

Let $\zeta_{\alpha,\gamma}$ denote the first-quadrant saddle and define $H_\beta(\alpha):=\exp(\Re\psi_{\alpha,\gamma}(\zeta_{\alpha,\gamma}))$ and $\Theta_\beta(\alpha):=\Im\psi_{\alpha,\gamma}(\zeta_{\alpha,\gamma})$.  Throughout the rest of the analysis, $K\Subset I$ means that $K$ is compactly contained in the interval $I$, and constants implicit in $O_K(\cdot)$ may depend on $K$ but are uniform for $\alpha\in K$.  On every compact $K\Subset(0,\alpha_c(\gamma))$, the saddle curve stays a positive distance from $0$, $\pm1$, and the coalescence point.  Hence the phase and amplitude form a uniformly analytic family on fixed saddle neighborhoods, and \eqref{eq:second-derivative} is uniformly bounded away from zero.

\begin{lemma}[Subcritical saddle action and phase]\label{lem:H-less-one}
For every $0<\alpha<\alpha_c(\gamma)$,
\begin{equation}\label{eq:H-range}
  0<H_\beta(\alpha)<1,
\end{equation}
and
\begin{equation}\label{eq:theta-phase-range}
  -\pi<\Theta_\beta(\alpha)<0.
\end{equation}
Moreover,
\begin{equation}\label{eq:H-endpoints}
  \lim_{\alpha\downarrow0}H_\beta(\alpha)=1,
  \qquad
  \lim_{\alpha\uparrow\alpha_c(\gamma)}H_\beta(\alpha)
  =\exp\!\left(
      \frac{2\zeta_c}{1+\zeta_c^2}
      -2\operatorname{artanh}\zeta_c
    \right)<1.
\end{equation}
\end{lemma}

\begin{proof}
Let $h_\beta(\alpha)=\log H_\beta(\alpha)$.  Stationarity of the saddle gives the envelope identities
\begin{equation}\label{eq:hprime-early}
  h_\beta'(\alpha)=\Re(\zeta_{\alpha,\gamma}^{\gamma}),
  \qquad
  \Theta_\beta'(\alpha)=\Im(\zeta_{\alpha,\gamma}^{\gamma})>0.
\end{equation}
As $\alpha\downarrow0$, the saddle tends to infinity while $\alpha\zeta_{\alpha,\gamma}^{\gamma}\to0$ by \eqref{eq:saddle-eq}.  Also $(1-\zeta_{\alpha,\gamma})/(1+\zeta_{\alpha,\gamma})\to-1$ from the lower half-plane.  Hence $h_\beta(\alpha)\to0$ and $\Theta_\beta(\alpha)\to-\pi$.  At coalescence the saddle is positive real, so $\Theta_\beta(\alpha)\to0$.  The second identity in \eqref{eq:hprime-early} therefore proves \eqref{eq:theta-phase-range}.

At $\alpha=\alpha_c(\gamma)$, $\zeta_c=\sqrt{(\gamma-1)/(\gamma+1)}$ is real.  Using the saddle equation and $\gamma=(1+\zeta_c^2)/(1-\zeta_c^2)$ gives
\[
  h_\beta(\alpha_c(\gamma)^-)
  =\frac{2\zeta_c}{1+\zeta_c^2}
   +\log\frac{1-\zeta_c}{1+\zeta_c}
  =\frac{2\zeta_c}{1+\zeta_c^2}
   -2\operatorname{artanh}\zeta_c<0,
\]
because $\operatorname{artanh}x>x/(1+x^2)$ for $0<x<1$.  Finally, with $\zeta=re^{\ii\theta}$, the first identity in \eqref{eq:hprime-early} has the sign of $\cos(\gamma\theta)$.  Since $\theta$ decreases continuously from $\pi/(\gamma+1)$ to $0$, this derivative changes sign exactly once.  Thus $h_\beta$ first decreases from $0$ and then increases to the still negative endpoint value above, proving \eqref{eq:H-range} and \eqref{eq:H-endpoints}.
\end{proof}

\subsection{Compact-ratio contour topology and localization}\label{subsec:steepest-geometry}

A local saddle expansion is not enough: we also need a contour on which these saddles are globally dominant.  Only a compact-$\alpha$ statement is needed, so we make exactly that topology explicit and no more.  The global issue has three parts: identify the two ends of the descending thimble, rule out the negative-real branch cut as a competing endpoint, and then make the resulting saddle gap uniform when $\alpha$ ranges over a compact subinterval.  The next three lemmas perform precisely these tasks.  In the open upper half-plane use
\begin{equation}\label{eq:upper-phase}
  \psi^+_{\alpha,\gamma}(\zeta)
  :=\alpha\zeta^\gamma+\log(1-\zeta)-\log(1+\zeta),
\end{equation}
where all powers and logarithms are taken by their principal branches.  This is analytic for $\Im\zeta>0$ and agrees with the phase \eqref{eq:phase} near the first-quadrant saddle.  Write $u:=\Re\psi^+_{\alpha,\gamma}$ and $v:=\Im\psi^+_{\alpha,\gamma}$. The descending flow is
\begin{equation}\label{eq:gradient-flow}
  \frac{d\zeta}{dt}=-\overline{(\psi^+_{\alpha,\gamma})'(\zeta)}.
\end{equation}
Along every nonstationary trajectory,
\begin{equation}\label{eq:flow-properties}
  \frac{d}{dt}u(\zeta(t))
  =-\bigl|(\psi^+_{\alpha,\gamma})'(\zeta(t))\bigr|^2<0,
  \qquad
  \frac{d}{dt}v(\zeta(t))=0.
\end{equation}

The following elementary boundary estimate is useful for ruling out an apparent extra endpoint on the negative-real branch cut.

\begin{lemma}[Negative-real boundary barrier]\label{lem:negative-boundary-barrier}
For $0<r<1$,
\begin{equation}\label{eq:artanh-barrier}
  2\operatorname{artanh}r\ge \alpha_c(\gamma)r^\gamma.
\end{equation}
Consequently, for every $0<\alpha<\alpha_c(\gamma)$,
\begin{equation}\label{eq:negative-boundary-u}
  \Re\psi^+_{\alpha,\gamma}(-r+\ii0)>0,
  \qquad 0<r<1.
\end{equation}
\end{lemma}

\begin{proof}
For $0<t<1$, $t^{1-\gamma}/(1-t^2)$ has its unique minimum $\gamma\alpha_c(\gamma)/2$ at $t^2=(\gamma-1)/(\gamma+1)$. Hence $(1-t^2)^{-1}\ge(\gamma\alpha_c(\gamma)/2)t^{\gamma-1}$, and integration gives \eqref{eq:artanh-barrier}. On the upper side of $(-1,0)$, $\Re\psi^+_{\alpha,\gamma}(-r+\ii0)=2\operatorname{artanh}r+\alpha r^\gamma\cos(\pi\gamma)$. This is positive if $\cos(\pi\gamma)\ge0$; otherwise \eqref{eq:artanh-barrier} gives $\Re\psi^+_{\alpha,\gamma}(-r+\ii0)\ge(\alpha_c(\gamma)-\alpha|\cos(\pi\gamma)|)r^\gamma>0$.
\end{proof}

\begin{lemma}[Topology of the upper descending thimble]\label{lem:thimble-topology}
Fix $0<\alpha<\alpha_c(\gamma)$ and let $\zeta_{\alpha,\gamma}$ be the first-quadrant saddle.  The two descending separatrices of \eqref{eq:gradient-flow} leaving this saddle have distinct ends: one approaches $\zeta=1$ from the upper half-plane, and the other is unbounded with
\begin{equation}\label{eq:thimble-infinity-direction}
  \arg\zeta\longrightarrow {\pi}/{\gamma}.
\end{equation}
Along either branch, $u$ decreases strictly from $\log H_\beta(\alpha)$ to $-\infty$.  In particular, their union, oriented from $1$ through the saddle to infinity, is a steepest-descent contour on which the saddle is the unique maximizer of $u$.
\end{lemma}

\begin{proof}
Both descending branches lie in the critical level $v=\Theta_\beta(\alpha)$, with $-\pi<\Theta_\beta(\alpha)<0$, and start at the negative height $u(\zeta_{\alpha,\gamma})=\log H_\beta(\alpha)<0$ by Lemma~\ref{lem:H-less-one}. Standard continuation for the smooth planar ODE~\eqref{eq:gradient-flow} shows that a maximal forward branch can cease to exist only by approaching the boundary of the upper half-plane or by leaving every compact subset. A descending branch cannot terminate at a regular interior point.  Nor can it recur or form a closed orbit, because $u$ decreases strictly along every nonstationary trajectory.  If a forward branch stayed in a compact subset of the upper half-plane while remaining a positive distance from the critical point, then $|(\psi_{\alpha,\gamma}^+)'|$ would have a positive lower bound there and \eqref{eq:flow-properties} would force $u\to-\infty$, impossible on a compact set.  Thus any compact limiting set would contain a critical point; the uniqueness in Lemma~\ref{lem:saddle-pair}, together with the strictly smaller value of $u$ after leaving the saddle, excludes this possibility.

We next inspect the real boundary.  For $0<x<1$, the upper boundary value of $v$ is $0$, while for $x>1$ it is $-\pi$; neither equals $\Theta_\beta(\alpha)$.  On $-1<x<0$ the level $v=\Theta_\beta(\alpha)$ may intersect the boundary, but Lemma~\ref{lem:negative-boundary-barrier} gives $u>0$, whereas $u$ strictly decreases from a negative saddle value along a descending branch.  Thus such an intersection cannot be a descending endpoint.  For $x<-1$, $v(x+\ii0)=\alpha|x|^\gamma\sin(\pi\gamma)-\pi\le-\pi$, with equality only when $\gamma=2$, so this boundary is also excluded.  The branch point $0$ has boundary phase $0$, and $u\to+\infty$ at $-1$.  Hence the only possible finite boundary endpoint of a descending branch is $\zeta=1$, where $u\to-\infty$.

If a descending branch is unbounded, write $\zeta=re^{\ii\vartheta}$.  Since $v$ is constant and the logarithmic term in \eqref{eq:upper-phase} has bounded imaginary part, $\sin(\gamma\vartheta)\to0$. Any accumulation direction is therefore a multiple of $\pi/\gamma$.  The direction $\vartheta\to0$ is impossible because then $u\to+\infty$; when $\gamma=2$, the same argument excludes $\vartheta\to\pi$.  The only descending channel in the open upper half-plane is consequently \eqref{eq:thimble-infinity-direction}, along which $u\to-\infty$.

It remains to see that the two descending branches cannot share the same end.  Near $\zeta=1$, set $\chi=1-\zeta$; then $\psi^+_{\alpha,\gamma}(\zeta)=\log\chi+C_1+O(\chi)$ with $C_1\in\R$.  Thus the equation $v=\Theta_\beta(\alpha)$ has exactly one upper-half-plane arc approaching $1$, with $\arg\chi\to\Theta_\beta(\alpha)$.  At infinity, the descending solution of the same level equation is likewise unique: writing $\vartheta=\pi/\gamma+\delta$ gives
\[
  v=-\pi-\alpha\gamma r^\gamma\delta+o(1)+O(r^\gamma\delta^2),
\]
so the implicit function theorem yields a unique branch with
\[
  \delta(r)
  =-\frac{\Theta_\beta(\alpha)+\pi}{\alpha\gamma}\,r^{-\gamma}
   +o(r^{-\gamma}).
\]
(The other level arc at infinity lies near the positive real direction and has $u\to+\infty$, hence is ascending.)  Each descending end therefore supports at most one separatrix.  Since there are two descending separatrices and the only possible descending ends are $1$ and the channel \eqref{eq:thimble-infinity-direction}, one branch has each end.  Strict monotonicity of $u$ follows from \eqref{eq:flow-properties}.
\end{proof}

\begin{lemma}[Compact-ratio saddle localization]\label{lem:global-contour}
Let $K\Subset(0,\alpha_c(\gamma))$.  For every $\alpha\in K$, after a vanishing indentation at $0$, the contour $\ii\R$ in \eqref{eq:coeff-exact} can be deformed to conjugate contours $\Gamma_\pm(\alpha)$ through $\zeta_{\alpha,\gamma}$ and $\overline{\zeta_{\alpha,\gamma}}$, together with oppositely oriented copies of the real segment $(0,1)$; the latter cancel in the sum of the upper and lower integrals.  Along $\Gamma_+(\alpha)$, $\Re\psi_{\alpha,\gamma}$ has its unique maximum at $\zeta_{\alpha,\gamma}$, and similarly on $\Gamma_-(\alpha)$.  Moreover, fixed saddle neighborhoods $U_\pm(\alpha)$ and $c_K>0$ may be chosen so that, uniformly for $\alpha\in K$,
\begin{equation}\label{eq:global-gap}
  \Re\psi_{\alpha,\gamma}(\zeta)
  \le
  \Re\psi_{\alpha,\gamma}(\zeta_{\alpha,\gamma})-c_K
\end{equation}
away from those neighborhoods.  The unbounded ends satisfy $\arg\zeta\to\pm\pi/\gamma$ and have real phase tending to $-\infty$.
\end{lemma}

\begin{proof}
Consider first the upper half-plane and an integer $n\ge2$.  By Lemma~\ref{lem:thimble-topology}, the upper descending thimble is a curve from $1$ through $\zeta_{\alpha,\gamma}$ to infinity in the direction $\pi/\gamma$.  For $0<\varepsilon<1<R$, replace the endpoint of the upper imaginary ray near $0$ by the right semicircle $|\zeta|=\varepsilon$.  The exact integrand in \eqref{eq:coeff-exact} is bounded there, so this indentation contributes $O(\varepsilon)$.

Truncate the thimble at radius $R$ and join its outer endpoint to $\ii R$ by the shorter circular arc.  By \eqref{eq:thimble-infinity-direction} and the asymptotic formula for $\delta(r)$ in the preceding proof, for all sufficiently large $R$ this arc lies entirely in a closed angular sector on which
\[
  \Re(\zeta^\gamma)\le-c_\gamma R^\gamma
\]
for some $c_\gamma>0$ independent of $R$.  Indeed, if $1<\gamma<2$, its angles lie between $\pi/2$ and a number tending to $\pi/\gamma$ from below; if $\gamma=2$, both endpoint angles tend to $\pi/2$ and the shorter arc remains in a fixed neighborhood of that direction, where $\cos(2\arg\zeta)<0$.  Thus the large-arc integral tends to zero exponentially as $R\to\infty$.  Cauchy's theorem in the upper half-plane then deforms the upper imaginary ray to the real connector $(\varepsilon,1)$ followed by the truncated descending thimble.  The deformation does not cross the negative-real branch cut or the pole $-1$; at $\zeta=1$ the exact integrand is regular (and vanishes to order $n-1$ for $n\ge2$).

Conjugation gives the lower deformation.  The exact integrand has the same boundary value on $(0,1)$ from above and below, while the two real connectors have opposite orientations.  They therefore cancel exactly.  Letting first $R\to\infty$ and then $\varepsilon\downarrow0$ gives the asserted conjugate contour representation.

It remains to make the saddle gap uniform in $\alpha\in K$.  The saddle curve $\{\zeta_{\alpha,\gamma}:\alpha\in K\}$ is compact and uniformly nondegenerate.  The analytic Morse lemma with parameter therefore gives one $\eta>0$ and, for every $\alpha\in K$, a local analytic coordinate $u_\alpha$ on a neighborhood of the upper saddle such that
\[
  \psi^+_{\alpha,\gamma}(\zeta)
  =\psi^+_{\alpha,\gamma}(\zeta_{\alpha,\gamma})
   -\frac12u_\alpha(\zeta)^2,
\]
with the upper descending thimble represented locally by real $u_\alpha$.  Shrinking $\eta$ if necessary, these coordinates and neighborhoods may be chosen uniformly on $K$.  Take $U_+(\alpha)$ to be the part with $|u_\alpha|<\eta$, and set $U_-(\alpha):=\overline{U_+(\alpha)}$.  At the two exits from this neighborhood,
\[
  \Re\psi_{\alpha,\gamma}
  =\log H_\beta(\alpha)-\frac12\eta^2.
\]
By the strict decrease of the real phase along each outgoing branch in Lemma~\ref{lem:thimble-topology}, every subsequent point of the thimble satisfies the same inequality.  Thus \eqref{eq:global-gap} holds on the entire upper contour outside $U_+(\alpha)$ with, for example, $c_K=\eta^2/2$; conjugation gives the lower contour.  This gap estimate itself avoids taking limits of parameter-dependent thimbles.

For later integral estimates we also record the required uniform arclength control.  Use the endpoint expansion at $1$ and the uniform large-$|\zeta|$ sector estimate above to remove fixed endpoint and tail pieces.  To justify compactness of the remaining pieces, set
\[
 \eta_K:=\min_{\alpha\in K}\min\{\Theta_\beta(\alpha)+\pi,-\Theta_\beta(\alpha)\}>0,
 \qquad h_K:=\max_{\alpha\in K}\log H_\beta(\alpha)<0.
\]
The phase separation $\eta_K$ excludes approach to the positive real boundary away from $1$ and to $(-\infty,-1)$, uniformly on bounded sets. Near $0$ the real phase tends uniformly to zero, near $-1$ it tends to $+\infty$, and on compact subsegments of $(-1,0)$ the boundary barrier is uniformly positive. None can be approached by a descending trajectory with real phase at most $h_K$. The endpoint and large-radius expansions are uniform for $\alpha\in K$. The exit points from the uniform Morse neighborhoods depend continuously on $\alpha$, and standard continuous dependence of solutions of the smooth ODE~\eqref{eq:gradient-flow} on parameters and initial data shows that the intervening truncated thimble arcs vary continuously with $\alpha$ until they meet the chosen endpoint and tail cross-sections. Together with the uniform endpoint expansions, this shows that the union of these middle pieces over $\alpha\in K$ is a compact subset of the product of $K$ with the open upper half-plane. Removing the uniform Morse neighborhoods excludes the graph of the critical point corresponding to the \emph{same} parameter. Lemma~\ref{lem:saddle-pair} and compactness now give
\[
  |(\psi^+_{\alpha,\gamma})'(\zeta)|\ge d_K>0
\]
uniformly on these middle pieces.  If $d\ell$ is arclength along a descending trajectory, then \eqref{eq:flow-properties} gives $d\ell=-du/|(\psi^+_{\alpha,\gamma})'|$.  Since $u=\Re\psi^+_{\alpha,\gamma}$ ranges over a fixed bounded interval there, the middle-piece arclengths are uniformly bounded.
\end{proof}

\begin{lemma}[Uniform off-saddle integral bound]\label{lem:off-saddle-integral}
Let $K\Subset(0,\alpha_c(\gamma))$, and choose the saddle neighborhoods in Lemma~\ref{lem:global-contour}.  There exist $C_K,c_K>0$ and an integer $n_0\ge2$ such that, for $n\ge n_0$ and $\alpha\in K$,
\begin{equation}\label{eq:off-saddle-integral}
  \int_{\Gamma_\pm(\alpha)\setminus U_\pm(\alpha)}
  |A(\zeta)|\,
  \e^{n\Re\psi_{\alpha,\gamma}(\zeta)}\,|d\zeta|
  \le C_K\,
  \e^{n(\log H_\beta(\alpha)-c_K)}.
\end{equation}
Here the phase--amplitude product is interpreted by continuity at $\zeta=1$, where the exact integrand in \eqref{eq:coeff-exact} is regular.
\end{lemma}

\begin{proof}
It is enough to treat $\Gamma_+(\alpha)$.  Write $u=\Re\psi_{\alpha,\gamma}^+$ and let $h_\alpha=\log H_\beta(\alpha)$.  By Lemma~\ref{lem:global-contour}, there is $c_0>0$, uniform for $\alpha\in K$, such that
\[
  u\le h_\alpha-2c_0
\]
on $\Gamma_+(\alpha)\setminus U_+(\alpha)$.

We first record a uniform integrable majorant at one fixed exponent.  Near $\zeta=1$, with $\chi=1-\zeta$,
\[
  u=\log|\chi|+O_K(1),
  \qquad |A(\zeta)|=O(|\chi|^{-1}),
\]
and the thimble is a regular arc.  Hence $|A|\e^{n_0u}$ is uniformly integrable there for any fixed $n_0\ge2$.  On the unbounded end, Lemma~\ref{lem:thimble-topology} and its asymptotic direction give
\[
  u\le-c|\zeta|^\gamma,
  \qquad A(\zeta)=O(|\zeta|^{-2}),
\]
uniformly for $\alpha\in K$; the thimble is asymptotic to a ray, so the same fixed-exponent integral is uniformly finite.  On the remaining middle part, the uniform lower bound $|(\psi^+_{\alpha,\gamma})'|\ge d_K$ established in Lemma~\ref{lem:global-contour}, together with $d\ell=-du/|(\psi^+_{\alpha,\gamma})'|$, gives a uniform arclength bound.  Consequently
\begin{equation}\label{eq:fixed-exponent-majorant}
  \sup_{\alpha\in K}
  \int_{\Gamma_+(\alpha)\setminus U_+(\alpha)}
  |A(\zeta)|\e^{n_0u(\zeta)}|d\zeta|<\infty.
\end{equation}
For $n\ge n_0$ factor the integrand as
\[
  |A|\e^{n_0u}\,\e^{(n-n_0)u}
  \le |A|\e^{n_0u}
     \e^{(n-n_0)(h_\alpha-2c_0)}.
\]
Since $h_\alpha$ stays in a compact interval for $\alpha\in K$, \eqref{eq:fixed-exponent-majorant} implies \eqref{eq:off-saddle-integral}, after decreasing $c_0$ and changing the constant.  The lower contour is the complex conjugate.
\end{proof}

\begin{figure}[tbp]
\centering
\includegraphics[width=.98\textwidth]{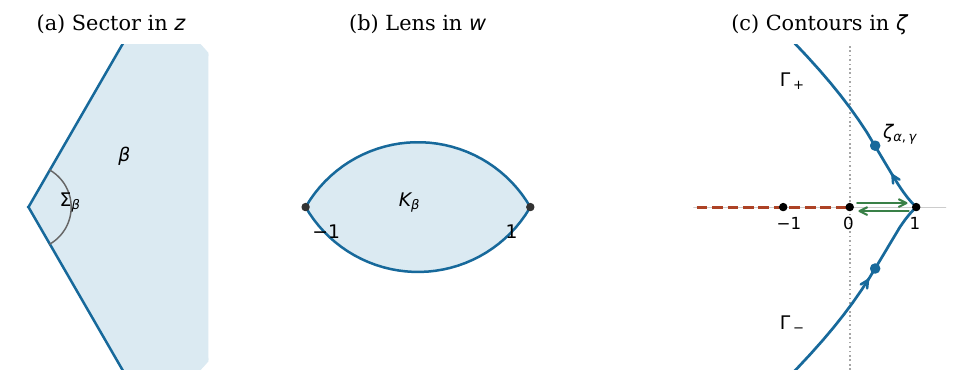}
\caption{Geometry at $\beta=2\pi/3$. Left and middle: the M\"obius image of $\Sigma_\beta$ is $K_\beta$; the unbounded sector in the left panel is shown schematically in a finite plotting window. Right: the deformed coefficient contours at $\alpha=\alpha_\beta^*$, the minimizer of the proved subcritical rate identified in Section~\ref{subsec:optimal-pole}. The upper contour is oriented from $1$ through the saddle to infinity; the lower contour has the reverse conjugate orientation. The real connectors cancel. The dashed negative axis is the branch cut.}
\label{fig:sector-contours}
\end{figure}

\begin{remark}\label{rem:contour-exact-integrand}
Figure~\ref{fig:sector-contours} summarizes the geometry. For noninteger $\gamma$, the indentation at $0$ is used to avoid the branch point of $\zeta^\gamma$; its contribution vanishes with the indentation radius. No indentation is needed when $\gamma=2$.  Lemma~\ref{lem:negative-boundary-barrier} excludes the extra descending endpoint created by the negative-real branch cut, while the two positive-real connectors cancel.  All contour statements here are for compact pole-ratio sets; the remote coefficient tail is treated separately in Section~\ref{sec:error-transfer}.
\end{remark}

\subsection{Uniform two-saddle asymptotics}\label{subsec:two-saddle}

\begin{theorem}[Two-saddle Faber coefficient asymptotics]\label{thm:two-saddle}
Let $K\Subset(0,\alpha_c(\gamma))$.  Uniformly for $\alpha\in K$, as $n\to\infty$,
\begin{equation}\label{eq:two-saddle-complex}
  a_n(\alpha n)
  =\frac{2}{\sqrt n}
  \Re\!\left\{
    \mathcal B_\beta(\alpha)
    \e^{n\psi_{\alpha,\gamma}(\zeta_{\alpha,\gamma})}
  \right\}
  +O_K\!\left(n^{-3/2}H_\beta(\alpha)^n\right),
\end{equation}
where $\mathcal B_\beta$ is a smooth nonvanishing complex amplitude.  With the square-root branch determined by the orientation of the upper steepest-descent contour, it may be written as
\begin{equation}\label{eq:complex-amplitude}
  \mathcal B_\beta(\alpha)
  =\frac{1}{\pi\ii(1-\zeta_{\alpha,\gamma}^2)}
   \left(\frac{2\pi}
   {-\psi''_{\alpha,\gamma}(\zeta_{\alpha,\gamma})}
   \right)^{1/2}.
\end{equation}

Equivalently, there is a smooth phase $\phi_\beta(\alpha)$ such that
\begin{equation}\label{eq:two-saddle-cos}
  a_n(\alpha n)
  =\frac{B_\beta(\alpha)}{\sqrt n}
   H_\beta(\alpha)^n
   \left[
   \cos\!\big(n\Theta_\beta(\alpha)+\phi_\beta(\alpha)\big)
   +O_K(n^{-1})
   \right],
\end{equation}
where
\begin{equation}\label{eq:B-explicit}
  B_\beta(\alpha)
  =\frac{2\sqrt{2\pi}}
  {\pi |1-\zeta_{\alpha,\gamma}^2|
   |\psi''_{\alpha,\gamma}(\zeta_{\alpha,\gamma})|^{1/2}}.
\end{equation}
In particular, $B_\beta$ is bounded above and away from zero on $K$.
\end{theorem}

\begin{proof}
The proof has two distinct ingredients: global localization and a local Gaussian calculation.  Lemmas~\ref{lem:global-contour} and \ref{lem:off-saddle-integral} supply the first one.  They replace the original imaginary axis by conjugate steepest-descent contours and show that the parts outside fixed saddle neighborhoods are exponentially smaller than $H_\beta(\alpha)^n$, uniformly for $\alpha\in K$.  Thus only the local geometry near the two saddles can contribute on the leading scale.

For an implementable orientation rule, write $\zeta_s=\zeta_{\alpha,\gamma}$ and $\psi''=\psi''_{\alpha,\gamma}$. Let $v_s$ be the unit tangent at the upper saddle pointing along the descending branch toward infinity. Of the two signs satisfying $v_s^2=-|\psi''(\zeta_s)|/\psi''(\zeta_s)$, choose the one whose outgoing trajectory reaches infinity. The square root in \eqref{eq:complex-amplitude} is then $\sqrt{2\pi}\,v_s/\sqrt{|\psi''(\zeta_s)|}$. Lemma~\ref{lem:thimble-topology} makes this choice unique and continuous in $\alpha$; no universal principal-root convention is assumed.

Consider the upper saddle $\zeta_s=\zeta_{\alpha,\gamma}$ and use the uniform Morse coordinate $u_\alpha$ already fixed in Lemma~\ref{lem:global-contour}.  Write $\tau:=u_\alpha(\zeta)$.  Shrinking its neighborhood uniformly if necessary, take the local thimble segment to correspond to the symmetric real interval $-\eta<\tau<\eta$, with
\[
  \tau(\zeta_s)=0,
  \qquad
  \psi_{\alpha,\gamma}(\zeta)
  =\psi_{\alpha,\gamma}(\zeta_s)-\frac12\tau^2.
\]
The orientation of the descending thimble selects the square-root branch for which
\[
  \left.\frac{d\zeta}{d\tau}\right|_{\tau=0}
  =\bigl(-\psi''_{\alpha,\gamma}(\zeta_s)\bigr)^{-1/2}.
\]
Hence the transformed prefactor
\[
  \frac{1}{\pi\ii}
  A(\zeta(\tau))\frac{d\zeta}{d\tau}
\]
is analytic in $\tau$, depends smoothly on $\alpha$, and at $\tau=0$ equals
\[
  \frac{1}{\pi\ii(1-\zeta_s^2)}
  \bigl(-\psi''_{\alpha,\gamma}(\zeta_s)\bigr)^{-1/2}.
\]
After the Gaussian rescaling $\tau=n^{-1/2}t$, its constant term therefore produces
\[
  \frac{1}{\sqrt n}\,
  \frac{1}{\pi\ii(1-\zeta_s^2)}
  \left(\frac{2\pi}{-\psi''_{\alpha,\gamma}(\zeta_s)}\right)^{1/2}
  \e^{n\psi_{\alpha,\gamma}(\zeta_s)}.
\]
The chosen local contour is symmetric in $\tau$, so odd Taylor terms integrate to zero and the next surviving term is smaller by $n^{-1}$. Uniform bounds on the local Taylor coefficients, together with the exponentially small off-saddle contribution, give an $O_K(n^{-3/2}H_\beta(\alpha)^n)$ remainder.  This identifies the upper-saddle coefficient as $\mathcal B_\beta(\alpha)$ in \eqref{eq:complex-amplitude}.

The lower contour has the reverse conjugate orientation. Its reversal cancels the sign change in $1/(\pi\ii)$ under conjugation, so its local contribution is the complex conjugate of the upper one.  Adding the two gives \eqref{eq:two-saddle-complex}.  Writing $\mathcal B_\beta=|\mathcal B_\beta|\e^{\ii\phi_\beta}$ and setting $B_\beta:=2|\mathcal B_\beta|$ gives \eqref{eq:two-saddle-cos}, while taking the modulus of \eqref{eq:complex-amplitude} yields \eqref{eq:B-explicit}.  Finally, compactness of $K$, separation from $\zeta=\pm1$, and \eqref{eq:second-derivative} show that $B_\beta$ is uniformly bounded above and away from zero.
\end{proof}

\begin{remark}\label{rem:why-shorter}
For $\alpha\in K\Subset(0,\alpha_c(\gamma))$, the saddles stay uniformly separated from singularities and coalescence, so the Gaussian scale $n^{-1/2}$ is uniform.
\end{remark}

\subsection{An index-action identity}\label{subsec:action-identity}

The next identity will control the near tail of omitted Faber coefficients in the coefficient-to-error argument.  Fix $q>0$ and, for real $x>q/\alpha_c(\gamma)$, set $\alpha_x:=q/x$ and $\zeta_x:=\zeta_{\alpha_x,\gamma}$, and define the saddle action
\begin{equation}\label{eq:action}
  S(x;q)
  :=x\,\psi_{q/x,\gamma}(\zeta_x).
\end{equation}

\begin{lemma}[Action derivative with respect to coefficient index]\label{lem:action-derivative}
For $x>q/\alpha_c(\gamma)$,
\begin{equation}\label{eq:action-derivative}
  \frac{\partial}{\partial x}S(x;q)
  =\log\frac{1-\zeta_x}{1+\zeta_x}.
\end{equation}
Consequently,
\begin{equation}\label{eq:action-decay}
  \frac{\partial}{\partial x}\Re S(x;q)
  =\log\left|\frac{1-\zeta_x}{1+\zeta_x}\right|<0.
\end{equation}
\end{lemma}

\begin{proof}
By definition,
\[
  S(x;q)
  =q\zeta_x^\gamma
   +x\log\frac{1-\zeta_x}{1+\zeta_x}.
\]
Differentiation gives
\begin{align*}
  \frac{\partial S}{\partial x}
  & =\log\frac{1-\zeta_x}{1+\zeta_x}
  +x\left[
     \frac{q}{x}\gamma\zeta_x^{\gamma-1}
     -\frac{2}{1-\zeta_x^2}
    \right]\frac{d\zeta_x}{dx}.
\end{align*}
The bracket vanishes by the saddle equation \eqref{eq:saddle-eq}, proving \eqref{eq:action-derivative}.  Since $\Re\zeta_x>0$, one has $|1-\zeta_x|<|1+\zeta_x|$, proving \eqref{eq:action-decay}.
\end{proof}

Identity \eqref{eq:action-decay} gives geometric tail decay while the effective ratio $q/x$ stays in a compact interval away from zero.

\section{From Faber coefficients to pole selection}\label{sec:error-transfer}

The coefficient theorem tells us the size of a typical omitted Faber coefficient, but the best uniform error depends on an entire tail and the coefficients oscillate.  This section closes that gap.  A remote-tail bound controls indices far from $m$; near $m$, a two-coefficient argument prevents complete cancellation.  Once these pieces are combined, the rate function can be optimized directly.

We return to degree $m$ and choose $q=\alpha m$ with $0<\alpha<\alpha_c(\gamma)$.  We transfer the oscillatory coefficient expansion to $E_{m,\beta}(\alpha m)$.  The saddle theorem controls the finite-ratio tail, while a separate estimate handles small effective ratios $q/n$, avoiding a fixed-to-moving saddle analysis.

\subsection{A crude bound for small effective ratios}\label{subsec:crude-tail}
The two-saddle theorem is uniform only while the effective ratio stays in a fixed compact subcritical interval.  Very remote tail coefficients have much smaller ratios, so we dispose of them first by a deliberately crude estimate; sharp constants are unnecessary there.

\begin{lemma}[Remote-tail coefficient bound]\label{lem:remote-bound}
There exist $\delta_\gamma,c_\gamma,C_\gamma>0$, depending only on $\gamma$, such that
\begin{equation}\label{eq:small-effective-ratio}
  0<\frac qn\le\delta_\gamma
\end{equation}
implies
\begin{equation}\label{eq:remote-bound}
  |a_n(q)|
  \le C_\gamma
  \exp\!\left[-c_\gamma q^{1/(\gamma+1)}n^{\gamma/(\gamma+1)}\right].
\end{equation}
\end{lemma}

\begin{proof}
The estimate comes from moving the contour far enough into the right half-plane that the rational factor supplies strong decay, but not so far that $\e^{q\zeta^\gamma}$ becomes too large.  On the line $\Re\zeta=a$, these two competing exponents have sizes $qa^\gamma$ and $n/a$.  Balancing them suggests $a\asymp(n/q)^{1/(\gamma+1)}$ and therefore the stretched-exponential scale $q^{1/(\gamma+1)}n^{\gamma/(\gamma+1)}$ appearing in \eqref{eq:remote-bound}.  The small constant $\kappa$ below tilts this balance toward decay.

Use the principal branch of $\zeta^\gamma$ in the right half-plane.  Split the contour in \eqref{eq:coeff-exact} at $0$, bypass the branch point by a right semicircle of radius $\varepsilon$, and shift the two pieces to $\Re\zeta=a$ with $a=\kappa(n/q)^{1/(\gamma+1)}$ and fixed sufficiently small $\kappa>0$.  The indentation is $O(\varepsilon)$ and disappears as $\varepsilon\downarrow0$.  No pole is crossed: $\zeta=1$ is regular for the exact integrand (and is a zero when $n\ge2$), while $\zeta=-1$ lies to the left.  The horizontal closing segments vanish uniformly as their height tends to infinity. Indeed, there are $M_\gamma,c_{\gamma,0}>0$ such that, whenever $|\Im\zeta|\ge M_\gamma a$ and $0\le\Re\zeta\le a$, the argument of $\zeta$ lies in a fixed neighborhood of $\pm\pi/2$ on which
\[
  \Re(\zeta^\gamma)\le-c_{\gamma,0}|\Im\zeta|^\gamma.
\]
Thus the exponential factor decays uniformly along the closing segments. Choose $\delta_\gamma$ so that \eqref{eq:small-effective-ratio} implies $a\ge2$.

Put $\zeta=a(1+\ii t)$.  Since $1<\gamma\le2$, constants $c_1,C_1>0$ exist such that
\begin{equation}\label{eq:power-upper}
  \Re(1+\ii t)^\gamma\le C_1-c_1|t|^\gamma,
  \qquad t\in\R.
\end{equation}
Furthermore,
\begin{align}\label{eq:mobius-modulus-bound}
  \log\left|\frac{1-\zeta}{1+\zeta}\right|
  &=\frac12\log\left(1-\frac{4a}{(1+a)^2+a^2t^2}\right) \le -\frac{c_2}{a(1+t^2)}
\end{align}
for a numerical $c_2>0$.  With $L=n/a$ one has $qa^\gamma=\kappa^{\gamma+1}L$, so the real part of the full exponent is at most
\begin{equation}\label{eq:scaled-far-exponent}
  L\left[
  C_1\kappa^{\gamma+1}
  -c_1\kappa^{\gamma+1}|t|^\gamma
  -\frac{c_2}{1+t^2}
  \right].
\end{equation}
For sufficiently small fixed $\kappa$, the bracket is bounded above by a negative constant uniformly in $t$.  Indeed, if $\chi_\kappa:=\kappa^{\gamma+1}$, the negative part contains
\[
  c_1\chi_\kappa|t|^\gamma+\frac{c_2}{1+t^2};
\]
it is bounded below by $c_2/2$ for $|t|\le1$, while for $|t|\ge1$ its minimum is of order $\chi_\kappa^{2/(\gamma+2)}$, which dominates the positive $O(\chi_\kappa)$ term as $\chi_\kappa\downarrow0$.

Finally $|d\zeta/(1-\zeta^2)|\le C[a(1+t^2)]^{-1}dt$, so integration gives $|a_n(q)|\le C\exp(-cL)$.  Since $L=\kappa^{-1}q^{1/(\gamma+1)}n^{\gamma/(\gamma+1)}$, \eqref{eq:remote-bound} follows after changing constants.
\end{proof}

\subsection{From coefficients to best error}\label{subsec:faber-sandwich}
The best error is trapped between two simple objects: a Faber projection gives an explicit upper bound, while any omitted coefficient gives a lower bound. The remaining issue is to choose coefficients that do not accidentally cancel.

For fixed $q>0$, define $a_0(q)$ by the same boundary integral as in \eqref{eq:faber-coeff-xi} with $n=0$.  The boundary coefficients then define the Faber series
\begin{equation}\label{eq:faber-series}
  g_q(w)=a_0(q)+\sum_{n=1}^\infty a_n(q)F_n(w).
\end{equation}
For fixed $q$, Lemma~\ref{lem:remote-bound} and $\|F_n\|_{K_\beta}\le2$ imply absolute uniform convergence of this series. The standard Faber reconstruction theorem of \citet[Lemma~4, p.~201]{KovariPommerenke1967} applies: $K_\beta$ is a closed Jordan domain with rectifiable boundary when $0<\beta<\pi$, and $g_q$ is analytic in its interior and continuous on its closure. With $F_n$ normalized as the polynomial part of $\Phi_\beta^n$, the boundary functional in~\eqref{eq:faber-coeff-xi} is the corresponding standard Faber-coefficient normalization. This reconstruction lemma requires only the uniform convergence already proved; no boundary convergence rate is imported. For $\beta=0$, uniqueness of the Fourier--Chebyshev expansion gives the same conclusion; see also \citet[Sec.~5]{Curtiss1971} for the Faber--Fourier connection. Define $\mathcal F_mg_q:=a_0(q)+\sum_{n=1}^m a_n(q)F_n$.

\begin{proposition}[Faber error sandwich]\label{prop:faber-sandwich}
For every $q>0$ and $m\ge0$,
\begin{equation}\label{eq:lower-coeff}
  E_{m,\beta}(q)\ge |a_n(q)|,
  \qquad n\ge m+1,
\end{equation}
and
\begin{equation}\label{eq:upper-tail}
  E_{m,\beta}(q)
  \le\|g_q-\mathcal F_mg_q\|_{K_\beta}
  \le2\sum_{n=m+1}^\infty|a_n(q)|.
\end{equation}
\end{proposition}

\begin{proof}
The boundary functional in \eqref{eq:faber-coeff-xi} has norm at most one in the supremum norm.  It annihilates every polynomial of degree $<n$: for such a polynomial $p$, the Laurent expansion of $p(\Psi_\beta(\xi))$ at infinity has no positive power $\xi^n$.  Applying this functional to $g_q-p$ therefore gives, for $p\in\Pm_m$ and $n\ge m+1$,
\[
  |a_n(q)|\le\|g_q-p\|_{K_\beta},
\]
and taking the infimum gives \eqref{eq:lower-coeff}.  The upper estimate follows from the uniformly convergent Faber expansion and \eqref{eq:faber-bound}.
\end{proof}

\medskip\noindent\textit{Summation of the coefficient tail.}\label{subsec:tail-sum}
We next sum the near tail on the correct exponential scale.  The action identity below shows that moving the coefficient index changes the saddle exponent in a controlled monotone way.

Let $K=[\alpha_-,\alpha_+]\Subset(0,\alpha_c(\gamma))$.  For $q=\alpha m$ and real $x\ge m$, recall $S(x;q)$ from \eqref{eq:action}.  If $x\le Tm$ with fixed $T>1$, then
\[
  \frac qx\in\left[\frac{\alpha_-}{T},\alpha_+\right]\Subset(0,\alpha_c(\gamma)),
\]
so Theorem~\ref{thm:two-saddle} applies uniformly.

\begin{lemma}[Uniform Faber-tail bound]\label{lem:tail-bound}
Let $K\Subset(0,\alpha_c(\gamma))$.  There exists $C_K>0$ such that, uniformly for $\alpha\in K$,
\begin{equation}\label{eq:tail-bound}
  \sum_{n=m+1}^\infty|a_n(\alpha m)|
  \le C_Km^{-1/2}H_\beta(\alpha)^m
\end{equation}
for all sufficiently large $m$.
\end{lemma}

\begin{proof}
Choose $T>1$ later.  On $m\le x\le Tm$, Lemma~\ref{lem:action-derivative} gives
\[
  \frac{d}{dx}\Re S(x;\alpha m)
  =\log\left|\frac{1-\zeta_{\alpha m/x,\gamma}}
                    {1+\zeta_{\alpha m/x,\gamma}}\right|.
\]
The effective ratios lie in a compact subinterval of $(0,\alpha_c(\gamma))$, so the right-hand side is at most $\log r_{K,T}$ for some $r_{K,T}<1$.  Therefore
\begin{equation}\label{eq:action-geometric}
  \Re S(n;\alpha m)
  \le m\log H_\beta(\alpha)+(n-m)\log r_{K,T},
  \qquad m\le n\le Tm.
\end{equation}
Theorem~\ref{thm:two-saddle}, used with coefficient index $n$ and effective ratio $\alpha m/n$, then yields
\begin{equation}\label{eq:finite-tail-bound}
  |a_n(\alpha m)|
  \le C_Km^{-1/2}H_\beta(\alpha)^m r_{K,T}^{\,n-m},
  \qquad m<n\le Tm,
\end{equation}
and the finite-ratio tail is summable geometrically.

Now choose $T$ so large that $\alpha_+/T\le\delta_\gamma$.  For $n>Tm$, Lemma~\ref{lem:remote-bound} gives
\[
  |a_n(\alpha m)|
  \le C_\gamma\exp\!\left[-c_\gamma(\alpha m)^{1/(\gamma+1)}n^{\gamma/(\gamma+1)}\right].
\]
An integral comparison for functions of the form $e^{-c x^p}$, with $c>0$ and $p=\gamma/(\gamma+1)$, shows that after decreasing $c_\gamma$ if necessary,
\begin{equation}\label{eq:remote-tail-sum}
  \sum_{n>Tm}|a_n(\alpha m)|
  \le C\exp\!\left[-c\alpha_-^{1/(\gamma+1)}T^{\gamma/(\gamma+1)}m\right].
\end{equation}
By Lemma~\ref{lem:H-less-one}, $\kappa_K:=\max_{\alpha\in K}\{-\log H_\beta(\alpha)\}<\infty$. Increasing $T$ once more so that the exponent in \eqref{eq:remote-tail-sum} exceeds $(\kappa_K+1)m$, the remote tail is $O(e^{-m}H_\beta(\alpha)^m)$ uniformly on $K$ and is absorbed into \eqref{eq:tail-bound}.
\end{proof}

\medskip\noindent\textit{A fixed two-coefficient noncancellation bound.}\label{subsec:noncancellation}
A single leading coefficient can be small because the conjugate saddle terms have nearly opposite phases.  Two consecutive coefficients, however, cannot both have that cancellation uniformly on a compact subcritical parameter set.  This elementary observation supplies the lower bound needed for best approximation.

Define
\begin{equation}\label{eq:omega}
  \omega_\beta(\alpha)
  :=\Arg\left(\frac{1-\zeta_{\alpha,\gamma}}
                    {1+\zeta_{\alpha,\gamma}}\right)\in(-\pi,0).
\end{equation}
On compact $K\Subset(0,\alpha_c(\gamma))$, this angle stays a positive distance from both endpoints.

\begin{lemma}[Two neighboring coefficients cannot both cancel]\label{lem:two-coeff}
Let $K\Subset(0,\alpha_c(\gamma))$.  There exists $c_K>0$ such that, uniformly for $\alpha\in K$ and all sufficiently large $m$,
\begin{equation}\label{eq:two-coeff-lower}
  \max\{|a_{m+1}(\alpha m)|,|a_{m+2}(\alpha m)|\}
  \ge c_Km^{-1/2}H_\beta(\alpha)^m.
\end{equation}
\end{lemma}

\begin{proof}
Apply Theorem~\ref{thm:two-saddle} with coefficient index $n=m+j$, $j=1,2$, and effective ratio $\alpha_{m,j}=\alpha m/(m+j)=\alpha+O_K(m^{-1})$. With $q=\alpha m$ and the action $S(x;q)$ from \eqref{eq:action}, separate the envelope from the oscillatory phase $\Im S(m+j;q)+\phi_\beta(\alpha_{m,j})$. The real action changes by only $O_K(1)$ when $x$ moves from $m$ to $m+2$, so both leading envelopes are uniformly comparable with $m^{-1/2}H_\beta(\alpha)^m$.

The decisive point is the phase increment.  By Lemma~\ref{lem:action-derivative},
\begin{align*}
  \Im\{S(m+2;q)-S(m+1;q)\}
  &=\int_{m+1}^{m+2}
    \Arg\!\left(\frac{1-\zeta_{q/x,\gamma}}
                       {1+\zeta_{q/x,\gamma}}\right)dx \\
  &=\omega_\beta(\alpha)+O_K(m^{-1}),
\end{align*}
because $q/x=\alpha+O_K(m^{-1})$ throughout this unit interval.  The smooth amplitude phase contributes only another $O_K(m^{-1})$.  Hence the two leading cosines are separated by
\begin{equation}\label{eq:phase-increment}
  \omega_\beta(\alpha)+O_K(m^{-1}).
\end{equation}
For real $\varphi,\omega$, $\cos^2\varphi+\cos^2(\varphi+\omega)\ge1-|\cos\omega|$.  On the compact set $K$, the angle $\omega_\beta(\alpha)$ stays a positive distance from $0$ and $-\pi$, so $\sup_{\alpha\in K}|\cos\omega_\beta(\alpha)|<1$.  Therefore one of the two leading cosine factors is uniformly bounded away from zero.  The normalized $O(m^{-1})$ remainder in Theorem~\ref{thm:two-saddle} is then absorbed.  This argument never divides by a single cosine that may itself be close to zero; that is exactly why two neighboring coefficients are used.
\end{proof}

\medskip\noindent\textit{Prefactor-resolved best-error bounds.}\label{subsec:best-error-strong}
We can now combine the upper and lower mechanisms.  The result keeps the Gaussian $m^{-1/2}$ scale, not only the $m$th-root factor.

\begin{theorem}[Prefactor-resolved order of the best repeated-pole error]\label{thm:best-error}
Let $K\Subset(0,\alpha_c(\gamma))$.  There are $0<c_K\le C_K<\infty$ such that, uniformly for $\alpha\in K$,
\begin{equation}\label{eq:best-error-strong}
  c_Km^{-1/2}H_\beta(\alpha)^m
  \le E_{m,\beta}(\alpha m)
  \le C_Km^{-1/2}H_\beta(\alpha)^m,
\end{equation}
for all sufficiently large $m$.  Consequently,
\begin{equation}\label{eq:best-root-limit}
  \lim_{m\to\infty}E_{m,\beta}(\alpha m)^{1/m}
  =H_\beta(\alpha),
  \qquad 0<\alpha<\alpha_c(\gamma).
\end{equation}
\end{theorem}

\begin{proof}
The upper bound follows from Proposition~\ref{prop:faber-sandwich} and Lemma~\ref{lem:tail-bound}.  For the lower bound, Proposition~\ref{prop:faber-sandwich} gives $E_{m,\beta}(\alpha m)\ge\max\{|a_{m+1}(\alpha m)|,|a_{m+2}(\alpha m)|\}$, and Lemma~\ref{lem:two-coeff} applies.  Taking $m$th roots proves \eqref{eq:best-root-limit}.
\end{proof}

By Proposition~\ref{prop:isometry}, Theorem~\ref{thm:best-error} is simultaneously a prefactor-resolved repeated-pole rational approximation result on the original infinite sector. At $\beta=0$, its two-sided order coincides with the exact-linear specialization proved in the companion half-line paper~\citep{XueZhang2026}; the genuinely sectorial content here is the extension to $\beta>0$.

\subsection{Minimization of the subcritical rate function}\label{subsec:optimal-pole}
With the best-error law established, pole selection becomes a one-variable calculus problem.  The saddle parametrization makes the minimizer and the minimum value explicit.

Set $h_\beta(\alpha):=\log H_\beta(\alpha)$.  The envelope identity already used in Lemma~\ref{lem:H-less-one} gives $h_\beta'(\alpha)=\Re(\zeta_{\alpha,\gamma}^{\gamma})$.

\begin{proposition}[Unique minimizer of the subcritical rate]\label{prop:optimal-rate}
On $0<\alpha<\alpha_c(\gamma)$, $H_\beta(\alpha)$ has the unique minimizer
\begin{equation}\label{eq:alpha-star}
  \alpha_\beta^*
  =\frac{1}{\gamma\sin(\pi/(2\gamma))}.
\end{equation}
At this point $\zeta_{\alpha_\beta^*,\gamma}=\exp(\ii\pi/(2\gamma))$, and
\begin{equation}\label{eq:rho-star}
  \rho_\beta^*
  :=H_\beta(\alpha_\beta^*)
  =\tan\!\left(\frac{\pi}{4\gamma}\right).
\end{equation}
Moreover,
\begin{equation}\label{eq:astar-range}
  0<\alpha_\beta^*<1<\alpha_c(\gamma).
\end{equation}
\end{proposition}

\begin{proof}
Write $\zeta=re^{\ii\theta}$ using \eqref{eq:rtheta}--\eqref{eq:alpha-theta}.  Then
\[
  h_\beta'(\alpha)=r^\gamma\cos(\gamma\theta).
\]
Since $\theta$ decreases strictly as $\alpha$ increases, this vanishes exactly at
\[
  \theta_*:=\frac{\pi}{2\gamma}.
\]
This lies in $(0,\pi/(\gamma+1))$.  Formula \eqref{eq:rtheta} gives $r_*=1$, and substitution into the saddle equation yields \eqref{eq:alpha-star}.  At the minimizer, $\alpha\zeta^\gamma$ is purely imaginary and $|(1-e^{\ii\theta_*})/(1+e^{\ii\theta_*})|=\tan(\theta_*/2)$, which proves \eqref{eq:rho-star}.  The sign change of $\cos(\gamma\theta)$ gives uniqueness.  Finally, concavity of sine on $(0,\pi/2)$ gives $\sin(\pi/(2\gamma))>1/\gamma$, hence $\alpha_\beta^*<1$. For the other inequality, let $M_\gamma:=\max_{0<t<1}t^{\gamma-1}(1-t^2)<1$. The computation in Lemma~\ref{lem:negative-boundary-barrier} gives $\alpha_c(\gamma)=2/(\gamma M_\gamma)$, so $\alpha_c(\gamma)>2/\gamma\ge1$.
\end{proof}

Thus the rate-minimizing physical pole within the noncoalescing linear regime is $s_{m,\beta}^{\mathrm{rate}}=-\alpha_\beta^*m/h$.  For $\beta=0$, Andersson's half-line theory makes this the global root-asymptotic repeated-pole optimum, with $\alpha_0^*=1/\sqrt2$ and $\rho_0^*=\sqrt2-1$.
For $\beta>0$, Theorem~\ref{thm:best-error} and Proposition~\ref{prop:optimal-rate} concern the noncoalescing branch $0<\alpha<\alpha_c(\gamma)$.  We do not analyze the coalescence transition or the supercritical branches, and hence do not use the stronger global-optimality terminology there.  On the proved branch the gap at coalescence is already strict: $H_\beta(\alpha_c(\gamma)^-)>\rho_\beta^*$, since $H_\beta$ is strictly increasing for $\alpha_\beta^*<\alpha<\alpha_c(\gamma)$.  This distinction does not affect the pole-design result near \eqref{eq:alpha-star}, which lies strictly inside the proved range; Section~\ref{subsec:numerical-rate} gives a finite-degree diagnostic beyond coalescence.

As $\beta\uparrow\pi$, $\gamma\downarrow1$, $\alpha_\beta^*\uparrow1$, and $\rho_\beta^*\uparrow1$.  Thus the classical angle-independent choice $q=m$ approaches the minimizer of the subcritical rate function as the sector approaches the right half-plane.

\begin{proposition}[Quadratic flatness at the minimizer]\label{prop:curvature}
Let $\theta_*=\pi/(2\gamma)$.  Then
\begin{equation}\label{eq:curvature}
  h_\beta''(\alpha_\beta^*)
  =\frac{\gamma^2\sin^2\theta_*\cos\theta_*}
  {\cos^2\theta_*+\gamma^2\sin^2\theta_*}>0.
\end{equation}
Consequently,
\begin{equation}\label{eq:local-rate-expansion}
  \log\frac{H_\beta(\alpha)}{\rho_\beta^*}
  =\frac12h_\beta''(\alpha_\beta^*)
   (\alpha-\alpha_\beta^*)^2
   +O((\alpha-\alpha_\beta^*)^3).
\end{equation}
\end{proposition}

\begin{proof}
Let $\zeta=\zeta_{\alpha,\gamma}$.  Implicit differentiation of \eqref{eq:saddle-eq} gives
\begin{equation}\label{eq:zeta-alpha-derivative}
  \zeta'(\alpha)
  =-\frac{\zeta(1-\zeta^2)}
  {\alpha[(\gamma-1)-(\gamma+1)\zeta^2]}.
\end{equation}
Since $h_\beta'(\alpha)=\Re(\zeta^\gamma)$, $h_\beta''(\alpha)=\Re(\gamma\zeta^{\gamma-1}\zeta'(\alpha))$. At the minimizer, $\zeta^\gamma=\ii$, $\alpha_\beta^*=1/(\gamma\sin\theta_*)$, and
\[
  (\gamma-1)-(\gamma+1)\e^{2\ii\theta_*}
  =-2\e^{\ii\theta_*}(\cos\theta_*+\ii\gamma\sin\theta_*).
\]
Substitution into \eqref{eq:zeta-alpha-derivative} gives
\[
  h_\beta''(\alpha_\beta^*)
  =\Re\!\left(
     \frac{\gamma\sin\theta_*}
     {\alpha_\beta^*(\cos\theta_*+\ii\gamma\sin\theta_*)}
   \right)
  =\frac{\gamma^2\sin^2\theta_*\cos\theta_*}
  {\cos^2\theta_*+\gamma^2\sin^2\theta_*},
\]
which is \eqref{eq:curvature}.  Taylor expansion at the strict minimizer then gives \eqref{eq:local-rate-expansion}.
\end{proof}

\begin{corollary}[Finite-degree minimizers on compact subcritical windows]\label{cor:finite-degree-minimizers}
Let $K=[a,b]\Subset(0,\alpha_c(\gamma))$ with $\alpha_\beta^*\in(a,b)$.  For each integer $m\ge1$, choose any
\[
  \widehat\alpha_m\in
  \operatorname*{argmin}_{\alpha\in K}E_{m,\beta}(\alpha m).
\]
Then such a minimizer exists, and, as $m\to\infty$,
\begin{equation}\label{eq:finite-minimizer-rate}
  \widehat\alpha_m\longrightarrow\alpha_\beta^*,
  \qquad
  |\widehat\alpha_m-\alpha_\beta^*|=O_K(m^{-1/2}).
\end{equation}
No uniqueness of the finite-$m$ minimizer, and no assertion about minimization outside $K$, is required.
\end{corollary}

\begin{proof}
For fixed $m$, the map $\alpha\mapsto E_{m,\beta}(\alpha m)$ is continuous on $K$: the transformed targets depend continuously on $\alpha$ in the uniform norm, and distance to the fixed finite-dimensional space $\Pm_m$ is Lipschitz with respect to that norm.  Hence the minimum on $K$ is attained.  Let $c_K,C_K$ be the constants in Theorem~\ref{thm:best-error}.  For all sufficiently large $m$, minimality of $\widehat\alpha_m$ gives
\[
  c_Km^{-1/2}H_\beta(\widehat\alpha_m)^m
  \le E_{m,\beta}(\widehat\alpha_m m)
  \le E_{m,\beta}(\alpha_\beta^*m)
  \le C_Km^{-1/2}H_\beta(\alpha_\beta^*)^m.
\]
Therefore, with $h_\beta=\log H_\beta$,
\begin{equation}\label{eq:finite-minimizer-action-gap}
  0\le h_\beta(\widehat\alpha_m)-h_\beta(\alpha_\beta^*)
  \le \frac{1}{m}\log\frac{C_K}{c_K}.
\end{equation}
The unique minimum in Proposition~\ref{prop:optimal-rate} first implies $\widehat\alpha_m\to\alpha_\beta^*$.  Proposition~\ref{prop:curvature} then gives a neighborhood of $\alpha_\beta^*$ and a constant $c>0$ for which
$h_\beta(\alpha)-h_\beta(\alpha_\beta^*)\ge c|\alpha-\alpha_\beta^*|^2$.  Combining this with \eqref{eq:finite-minimizer-action-gap} proves \eqref{eq:finite-minimizer-rate}.
\end{proof}

\subsection{The simple choice \texorpdfstring{\(q=m\)}{q=m} and stable interpolation}\label{subsec:q-equals-m}
The subcritical rate-minimizing ratio depends on the sector angle, but in applications one may prefer the angle-independent rule $q=m$.  We quantify its local overhead and then record a simple transfer principle for stable interpolation.

By \eqref{eq:astar-range}, $q=m$ lies between the subcritical minimizer and coalescence for every $0\le\beta<\pi$, and Theorem~\ref{thm:best-error} gives $E_{m,\beta}(m)\asymp m^{-1/2}H_\beta(1)^m$.  Thus the theory provides a sharp best-approximation benchmark for the classical angle-independent choice $q=m$, while \eqref{eq:alpha-star} quantifies the additional rate improvement available from sector-dependent pole selection.  The root factors also give a more directly computational comparison.  For $0<\varepsilon<1$, define the first attainment degree $m_\varepsilon(\alpha):=\min\{m\in\mathbb N:m\ge1,\ E_{m,\beta}(\alpha m)\le\varepsilon\}$. Theorem~\ref{thm:best-error} makes this set nonempty.  For every fixed integer $m\ge1$, the finite-dimensional space $\Pm_m$ is closed and $g_{\alpha m}\notin\Pm_m$, so $E_{m,\beta}(\alpha m)>0$; hence $m_\varepsilon(\alpha)\to\infty$ as $\varepsilon\downarrow0$.  With $L=\log(1/\varepsilon)$ and $a_\alpha=-\log H_\beta(\alpha)>0$, the same two-sided bounds give, without any monotonicity assumption on the error sequence,
\[
  m_\varepsilon(\alpha)
  =\frac{L-\tfrac12\log L}{a_\alpha}+O(1).
\]
Indeed, the two-sided estimate is equivalent at threshold to $a_\alpha m+\tfrac12\log m=L+O(1)$, and one substitution gives the displayed formula.  Consequently,
\begin{equation}\label{eq:degree-ratio}
  \frac{m_\varepsilon(1)}{m_\varepsilon(\alpha_\beta^*)}
  \longrightarrow
  \frac{-\log\rho_\beta^*}{-\log H_\beta(1)},
  \qquad \varepsilon\downarrow0.
\end{equation}
Thus the lower-order attainment-degree correction is $-(2a_\alpha)^{-1}\log\log(1/\varepsilon)$.

\begin{corollary}[Subexponentially stable interpolation]\label{cor:stable-interpolation}
Let $I_m:C(K_\beta)\to\Pm_m$ be interpolation projectors, so that $I_mp=p$ for every $p\in\Pm_m$, with Lebesgue constants $\Lambda_m$ satisfying \(\lim_{m\to\infty}\Lambda_m^{1/m}=1\). Then, for every fixed $0<\alpha<\alpha_c(\gamma)$,
\begin{equation}\label{eq:interp-root-rate}
  \lim_{m\to\infty}
  \|g_{\alpha m}-I_mg_{\alpha m}\|_{K_\beta}^{1/m}
  =H_\beta(\alpha).
\end{equation}
\end{corollary}

\begin{proof}
The Lebesgue inequality gives $E_{m,\beta}(\alpha m)\le\|g_{\alpha m}-I_mg_{\alpha m}\|_{K_\beta}\le(1+\Lambda_m)E_{m,\beta}(\alpha m)$. Apply \eqref{eq:best-root-limit} and \(\lim_{m\to\infty}\Lambda_m^{1/m}=1\) to reach the conclusion.
\end{proof}

Thus subexponential stability preserves the sharp root factor, while a polynomially growing Lebesgue constant need not preserve the $m^{-1/2}$ prefactor.

\section{Numerical diagnostics and finite-degree behavior}\label{sec:numerics}

We first check the coefficient and error scales, then approximate the finite-degree pole selection problem. All computed errors and optimization gaps are finite-precision diagnostics, not certified enclosures. Section~\ref{subsec:matrix-experiment} tests the actual projected Arnoldi approximation.

\subsection{Coefficients, error scales, and numerical precision}\label{subsec:numerical-coefficients}

Writing $\xi=\e^{\ii\vartheta}$ in \eqref{eq:faber-coeff-xi} gives
\begin{equation}\label{eq:numerical-fourier-coeff}
 a_n(q)=\frac1{2\pi}\int_{-\pi}^{\pi}
 \exp\!\left(q\left(\ii\tan\frac{\vartheta}{2}\right)^\gamma\right)
 \e^{-\ii n\vartheta}\,d\vartheta.
\end{equation}
For noninteger $\gamma$ the integrand has a fractional-power cusp at zero, so the trapezoidal rule need not converge spectrally. The coefficient data reported here use adaptive split quadrature (SciPy/QUADPACK), with relative tolerance $2\times10^{-12}$ and absolute tolerance $10^{-14}$; the integration interval is split at the fractional-power cusp and at fixed nearby points. Independent $70$-digit tanh--sinh checks agree within $1.2\times10^{-12}$ relatively at the four $n=20$ cases below. At $\beta=2\pi/3$, $\alpha=1$, $m=42$, the two-coefficient quantity $\max\{|a_{43}(42)|,|a_{44}(42)|\}$ is $4.037395762576\times10^{-9}$ in that reference calculation: adaptive quadrature differs by $4.03\times10^{-10}$ relatively, whereas $2^{17}$- and $2^{18}$-point FFTs differ by $2.20\times10^{-3}$ and $4.36\times10^{-4}$. Thus an absolute quadrature tolerance alone is not an accuracy guarantee for a small oscillatory coefficient.

For this revision, all $84$ plotted coefficients were also checked on the shifted vertical contour $\Re\zeta=\Re\zeta_{\alpha,\gamma}$, factoring out $H_\beta(\alpha)^n$ before integration. The largest relative discrepancy was $2.89\times10^{-8}$; its effect on the normalized remainder below was less than $5.78\times10^{-7}$. This check includes the largest indices and the smallest coefficients, and uses a different contour from the boundary-quadrature calculation.

Let $a_n^{(0)}(\alpha n)$ be the leading term in \eqref{eq:two-saddle-complex}. Figure~\ref{fig:coefficient-remainder}(a) plots the remainder diagnostic
\begin{equation}\label{eq:remainder-diagnostic}
 R_n:=n^{3/2}H_\beta(\alpha)^{-n}|a_n(\alpha n)-a_n^{(0)}(\alpha n)|,
\end{equation}
which avoids division by a possibly vanishing oscillatory coefficient. We take $\beta=2\pi/3$ and $\alpha=0.3,\alpha_\beta^*,1,0.9\alpha_c(\gamma)$. The last ratio is fixed below coalescence. Its rise near $n=50$ is a finite-index effect: the shifted-contour values of $R_n$ at $n=100,200,400$ are $1.763,1.734,1.796$, respectively. These observations support, but do not prove, the $O(1)$ assertion of Theorem~\ref{thm:two-saddle}.

For the error comparison, define $L_m(\alpha):=\max\{|a_{m+1}(\alpha m)|,|a_{m+2}(\alpha m)|\}$. If $I_m^{\rm L}$ denotes the discrete-Leja interpolant, then
\[
 L_m(\alpha)\le E_{m,\beta}(\alpha m)
 \le\|g_{\alpha m}-I_m^{\rm L}g_{\alpha m}\|_{K_\beta}
\]
is rigorous. We plot the quadrature approximation $\widehat L_m$ and the independent-grid sampled interpolation error $\widehat U_m^{\rm L}$. The Leja sequence starts at $1$ and greedily maximizes the logarithmic distance product over $30001$ candidates, with $\pm1$ appended for the lens. An independent $50001$-point grid measures the error; the Leja construction uses a fixed deterministic ordering for ties. Figure~\ref{fig:strong-sandwich}(b) is consistent with the best-error scale. For Leja interpolation itself, Corollary~\ref{cor:stable-interpolation} preserves only the root factor under its stability hypothesis; the plot establishes neither that hypothesis nor the sharp algebraic prefactor.

\begin{figure}[tbp]
\centering
\includegraphics[width=.98\textwidth]{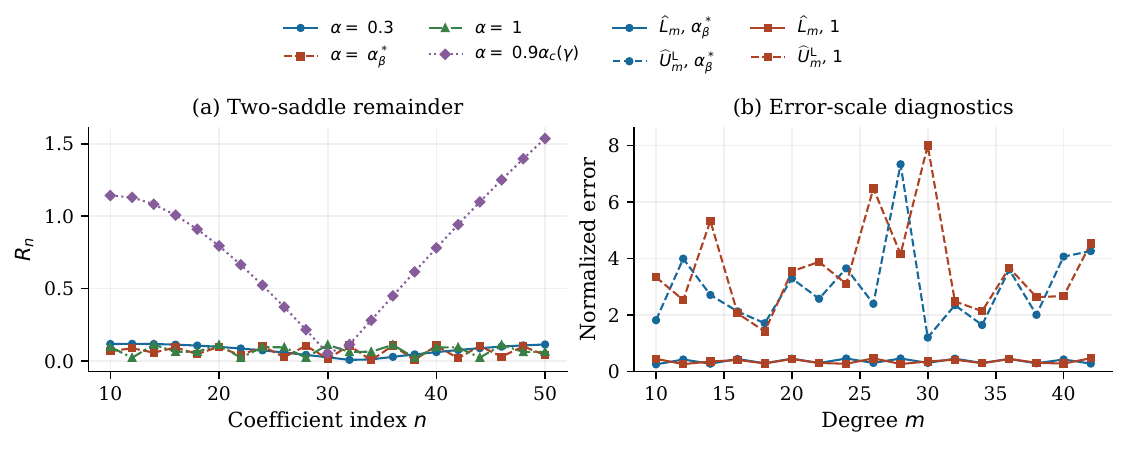}
\caption{Diagnostics at $\beta=2\pi/3$. (a) The normalized coefficient remainder $R_n$ for four ratios. (b) The computed two-coefficient quantity $\widehat L_m$ (solid) and sampled Leja error $\widehat U_m^{\rm L}$ (dashed), divided by $m^{-1/2}H_\beta(\alpha)^m$, at $\alpha=\alpha_\beta^*$ and $1$. The plotted quantities are not certified bounds.}
\label{fig:coefficient-remainder}\label{fig:strong-sandwich}
\end{figure}

\subsection{Asymptotic and finite-degree pole selection}\label{subsec:numerical-rate}

Figure~\ref{fig:rate-functions} compares the explicit rate functions and the sampled Leja root errors for $q=m$. Table~\ref{tab:rate-values} converts the rate difference into the leading small-tolerance degree overhead using \eqref{eq:degree-ratio}. The overhead is at most $2.30\%$ for the four listed angles; no all-angle monotonicity claim is intended.

\begin{figure}[tbp]
\centering
\includegraphics[width=.98\textwidth]{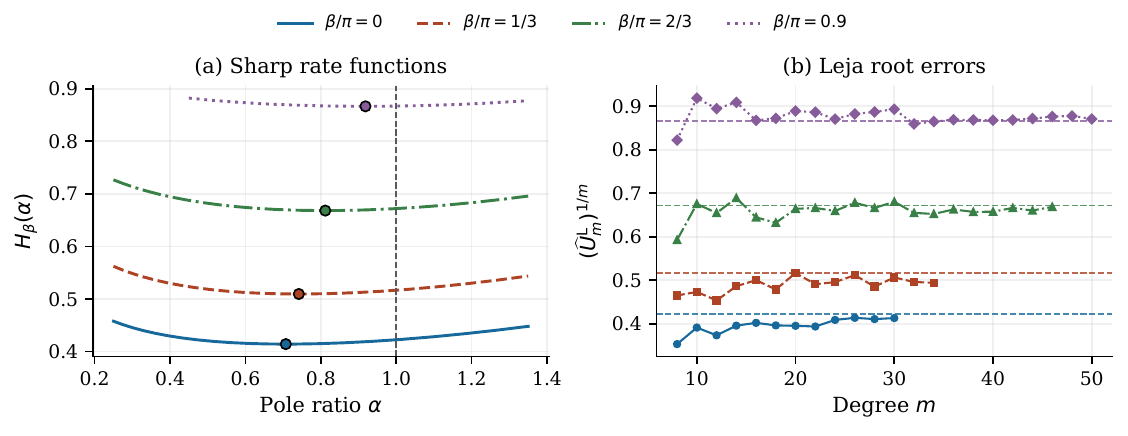}
\caption{(a) The rate functions $H_\beta(\alpha)$, with their exact minima marked by black-edged circles and the rule $q=m$ marked vertically. (b) Sampled Leja root errors $(\widehat U_m^{\rm L})^{1/m}$ for $q=m$; dashed levels are $H_\beta(1)$. The interval curve stops at $m=30$ to avoid the roundoff floor.}
\label{fig:rate-functions}\label{fig:leja-root}
\end{figure}

\begin{table}[tbp]
\centering
\caption{Geometry, sharp factors, and leading small-tolerance degree overhead of $q=m$. The overhead is $100[(-\log\rho_\beta^*)/(-\log H_\beta(1))-1]\%$, as given by~\eqref{eq:degree-ratio}.}
\label{tab:rate-values}
\begin{tabular}{c c c c c c}
\hline
$\beta/\pi$ & $\capc(K_\beta)$ & $\alpha_\beta^*$ & $\rho_\beta^*$ & $H_\beta(1)$ & overhead \\
\hline
$0$ & $0.500000$ & $0.707107$ & $0.414214$ & $0.422470$ & $2.29\%$ \\
$1/3$ & $0.600000$ & $0.741641$ & $0.509525$ & $0.516702$ & $2.12\%$ \\
$2/3$ & $0.750000$ & $0.811794$ & $0.668179$ & $0.671972$ & $1.42\%$ \\
$0.9$ & $0.909091$ & $0.918439$ & $0.866505$ & $0.866902$ & $0.32\%$ \\
\hline
\end{tabular}
\end{table}

To distinguish this rate minimization from finite-degree best approximation, we also solve a discrete complex minimax problem on the lens. Conjugation symmetry permits real polynomial coefficients: averaging a polynomial with $\overline{p(\overline w)}$ cannot increase its error. The maximum modulus principle reduces the uniform error to the lens boundary, since $g_{\alpha m}-p$ is analytic inside and continuous on the closure. On boundary points $w_j$, we minimize $t$ subject to $|g_{\alpha m}(w_j)-p(w_j)|\le t$, using a discrete orthogonal polynomial basis. Supporting half-planes of the residual disks give linear-programming relaxations; adding the active residual directions closes the discrete lower--upper gap. An independent constrained optimizer checks the selected solutions.

For $\beta=\pi/3,2\pi/3$ and $m=8,12,16,20$, we scan $33$ ratios in $[0.35,1.35]$, refine each detected local minimum, and repeat with $513$ and $1025$ upper-boundary nodes. A $2049$-node solve refines the selected minimum. Let $\widetilde\alpha_m$ denote this computed selection, and let $\widetilde E_m(\alpha)$ be the resulting polynomial's largest residual on a separate $16385$-point grid with local maximization around each detected peak. This is a diagnostic for continuous minimax error, not $E_{m,\beta}(\alpha m)$ itself or a certified global search in $\alpha$. The selected ratios change by less than $6\times10^{-6}$ under grid refinement; the largest relative validation-to-LP gap is $4.4\times10^{-5}$. The independent optimizer agrees within $4\times10^{-9}$ relatively on its training grid.

\input{pole_selection_table.tex}

Table~\ref{tab:finite-poles} shows that the finite-degree selection can differ appreciably from $\alpha_\beta^*$. The local minima of the oscillatory objective compete, so the selected ratios need not approach $\alpha_\beta^*$ monotonically. The displayed $\sqrt m\,|\widetilde\alpha_m-\alpha_\beta^*|$ values are consistent with the $O(m^{-1/2})$ localization bound in Corollary~\ref{cor:finite-degree-minimizers}, but these modest degrees do not determine a sharp asymptotic rate. These finite-degree effects also explain why the asymptotic overhead in Table~\ref{tab:rate-values} need not predict a small-degree comparison.

Because the theory deliberately stops at saddle coalescence, we also made a simple degree-$20$ diagnostic beyond it. Proposition~\ref{prop:faber-sandwich} gives the exact lower bound
\[
  L_{20}(\alpha):=\max\{|a_{21}(20\alpha)|,|a_{22}(20\alpha)|\}
  \le E_{20,\beta}(20\alpha).
\]
Sampling $21$ equally spaced ratios on $[\alpha_c(\gamma),3]$, we found that, among these sampled ratios, the computed $L_{20}$ was smallest at the coalescence endpoint for both tested angles. At $\beta=\pi/3$, $L_{20}(\alpha_c)=1.7596\times10^{-4}$, more than $3.1\times10^3$ times the selected subcritical diagnostic in Table~\ref{tab:finite-poles}; at $\beta=2\pi/3$, $L_{20}(\alpha_c)=6.4646\times10^{-3}$, about $2.47\times10^2$ times larger. The coalescence-point coefficients were independently recomputed on a shifted vertical contour, agreeing with the boundary quadrature to better than $6\times10^{-13}$ relatively. No monotonicity between the sampled ratios, or beyond $\alpha=3$, is inferred. These are finite-degree numerical diagnostics, not a theorem about the supercritical asymptotics, and they do not replace the restriction $0<\alpha<\alpha_c(\gamma)$ in the results above.

\section{Matrix and shift-and-invert Krylov consequences}\label{sec:matrix}

No new saddle asymptotics are needed in this section.  The scalar sector theorem transfers through the numerical range to dimension-independent bounds for nonnormal matrices, the worst case over normal matrices recovers the scalar error exactly, and the same M\"obius algebra identifies the repeated-pole class with one shift-and-invert Krylov space.

Fix $0\le\beta<\pi$, $h>0$, and $0<\alpha<\alpha_c(\gamma)$.  For an integer degree $m\ge1$, let $q_m=\alpha m$, choose a best $p_m^*\in\Pm_m$ in \eqref{eq:best-error}, and set $r_m^*(z):=p_m^*((q_m-hz)/(q_m+hz))$.  Then $\|\e^{-hz}-r_m^*\|_{\Sigma_\beta}=E_{m,\beta}(\alpha m)$.

\subsection{Matrix-function transfer for matrices with sectorial numerical range}\label{subsec:matrix-transfer}

For $A\in\C^{N\times N}$, write $W(A):=\{v^*Av:v\in\C^N,\ \|v\|_2=1\}$ for the numerical range.  Numerical-range containment, rather than spectral containment alone, is the natural hypothesis for the following transfer to nonnormal matrices.

\begin{proposition}[Numerical-range transfer and worst-case bound]\label{prop:numerical-range-transfer}
Let $A\in\C^{N\times N}$ satisfy \(W(A)\subset\Sigma_\beta\). Then
\begin{equation}\label{eq:crouzeix-transfer}
  \|\e^{-hA}-r_m^*(A)\|_2
  \le (1+\sqrt2)E_{m,\beta}(\alpha m).
\end{equation}
Moreover, with
\[
  \mathcal E_{m,\beta}^{W}(\alpha)
  :=\sup_{N\ge1}\
  \sup_{\substack{A\in\C^{N\times N}\\ W(A)\subset\Sigma_\beta}}
  \|\e^{-hA}-r_m^*(A)\|_2,
\]
one has
\begin{equation}\label{eq:numerical-range-worstcase}
  E_{m,\beta}(\alpha m)
  \le \mathcal E_{m,\beta}^{W}(\alpha)
  \le (1+\sqrt2)E_{m,\beta}(\alpha m).
\end{equation}
Hence, uniformly for $\alpha$ in every compact $K\Subset(0,\alpha_c(\gamma))$,
\begin{equation}\label{eq:nonnormal-strong-upper}
  \mathcal E_{m,\beta}^{W}(\alpha)
  \asymp m^{-1/2}H_\beta(\alpha)^m,
  \qquad
  \lim_{m\to\infty}
  \bigl(\mathcal E_{m,\beta}^{W}(\alpha)\bigr)^{1/m}
  =H_\beta(\alpha).
\end{equation}
For each fixed $A$ satisfying \(W(A)\subset\Sigma_\beta\), the corresponding individual error has $\limsup$ root at most $H_\beta(\alpha)$.
\end{proposition}

\begin{proof}
The pole of $r_m^*$ is $-q_m/h<0$, which lies outside $\Sigma_\beta$ because $\beta<\pi$.  Thus $f_m(z):=\e^{-hz}-r_m^*(z)$ is analytic on a neighborhood of the compact numerical range $W(A)$. The Crouzeix--Palencia inequality \citep{CrouzeixPalencia2017} gives $\|f_m(A)\|_2\le(1+\sqrt2)\|f_m\|_{W(A)}\le(1+\sqrt2)E_{m,\beta}(\alpha m)$, proving the upper bound in \eqref{eq:numerical-range-worstcase}.

For the lower bound, the class in \eqref{eq:numerical-range-worstcase} contains all scalar matrices $A=[z]$ with finite $z\in\Sigma_\beta$.  For every $\varepsilon>0$ there is such a finite $z_\varepsilon$ with $|\e^{-hz_\varepsilon}-r_m^*(z_\varepsilon)|\ge E_{m,\beta}(\alpha m)-\varepsilon$, even if the scalar supremum is attained only at the compactification point $\infty$.  Letting $\varepsilon\downarrow0$ proves the lower bound.  The two-sided order and root limit then follow from Theorem~\ref{thm:best-error}.
\end{proof}

For normal matrices, the bound improves from the Crouzeix--Palencia constant to the exact spectral-theorem value.  Indeed, if $A$ is normal and $\sigma(A)\subset\Sigma_\beta$, then
\[
  \|\e^{-hA}-r_m^*(A)\|_2
  =
  \max_{\lambda\in\sigma(A)}
  |\e^{-h\lambda}-r_m^*(\lambda)|
  \le E_{m,\beta}(\alpha m).
\]
Conversely, scalar matrices are normal, so the sector supremum is the worst case over all normal matrices.  Consequently,
\[
  \sup_{N\ge1}\
  \sup_{\substack{A\in\C^{N\times N}\ {\rm normal}\\
                  \sigma(A)\subset\Sigma_\beta}}
  \|\e^{-hA}-r_m^*(A)\|_2
  =
  E_{m,\beta}(\alpha m).
\]
Thus the worst-case normal-matrix error equals the scalar best error exactly and therefore inherits the same prefactor-resolved two-sided order, with no additional matrix constant.

\begin{remark}[Sector-uniform versus matrix-specific]\label{rem:matrix-specific}
The worst-case result is sector-uniform.  A particular matrix or starting vector may see faster convergence if the numerical range occupies only a small part of the sector; Arnoldi-dependent residual estimates can exploit this additional information \citep{MoretNovati2004,BeckermannGuettel2012}.
\end{remark}

\subsection{The shift-and-invert Krylov subspace}\label{subsec:shift-invert-space}

The M\"obius correspondence also identifies the rational search space exactly.  For a matrix $A$ with $W(A)\subset\Sigma_\beta$ and for $q>0$, set $B_q:=(qI+hA)^{-1}$ and $Z_q:=(qI-hA)(qI+hA)^{-1}=2qB_q-I$.  The inverse exists because $-q/h\notin W(A)$.

\begin{proposition}[One repeated pole equals one shift-and-invert Krylov space]\label{prop:space-equality}
For every nonzero vector $b$ and every $m\ge0$,
\begin{equation}\label{eq:space-equality}
  \{r(A)b:r\in\Rm_m(q;h)\}
  =\operatorname{span}\{b,B_qb,B_q^2b,\ldots,B_q^mb\}
  =:\mathcal K_{m+1}(B_q,b).
\end{equation}
\end{proposition}

\begin{proof}
By definition, every $r\in\Rm_m(q;h)$ has the form $r(A)=p(Z_q)$ with $p\in\Pm_m$.  Since $Z_q=2qB_q-I$, polynomials of degree at most $m$ in $Z_q$ are exactly polynomials of degree at most $m$ in $B_q$.  Applying them to $b$ proves \eqref{eq:space-equality}.
\end{proof}

Thus $q_m=\alpha m$ requires repeated applications of the single inverse $(q_mI+hA)^{-1}$, reusing one factorization or preconditioner. Combining Proposition~\ref{prop:space-equality} with Proposition~\ref{prop:numerical-range-transfer} gives the following search-space bound.

\begin{corollary}[Dimension-independent shift-and-invert approximation bound]\label{cor:shift-invert-bound}
Let $K\Subset(0,\alpha_c(\gamma))$.  If $m\ge1$, $W(A)\subset\Sigma_\beta$, $\alpha\in K$, and $q_m=\alpha m$, then there exists $C_{K,\beta}>0$, independent of $m$, $A$, and $b$, such that
\begin{equation}\label{eq:subspace-distance}
  \dist\!\left(
     \e^{-hA}b,
     \mathcal K_{m+1}((q_mI+hA)^{-1},b)
  \right)
  \le C_{K,\beta}\,\|b\|_2\,
       m^{-1/2}H_\beta(\alpha)^m.
\end{equation}
Here $\dist$ denotes Euclidean distance to the indicated subspace.  For normal $A$ with $\sigma(A)\subset\Sigma_\beta$, the same conclusion holds with the Crouzeix--Palencia factor removed from the underlying scalar-to-matrix estimate.
\end{corollary}

\begin{proof}
The vector $r_m^*(A)b$ belongs to the Krylov space by Proposition~\ref{prop:space-equality}.  Therefore the distance to that space is at most $\|\e^{-hA}b-r_m^*(A)b\|_2\le\|\e^{-hA}-r_m^*(A)\|_2\,\|b\|_2$. Apply Proposition~\ref{prop:numerical-range-transfer} and Theorem~\ref{thm:best-error}.
\end{proof}

The normal-matrix improvement in Corollary~\ref{cor:shift-invert-bound} is a search-space statement. It does not by itself remove the Crouzeix--Palencia factor from the projected bound below, because a compression of a normal transformed operator need not itself be normal unless additional structure is present.

For the standard transformed-variable RD--Arnoldi projection used in \citet{MoretNovati2004,Novati2011}, one obtains an equally direct algorithmic bound.  Let $b\ne0$, and, assuming no earlier Arnoldi breakdown, let $V_{m+1}$ be an orthonormal basis of $\mathcal K_{m+1}(B_q,b)$ generated by Arnoldi, with first vector $b/\|b\|_2$.  (If breakdown occurs earlier, the Krylov space is invariant and the projected function action is exact.)  Put $H_{m+1}:=V_{m+1}^*B_qV_{m+1}$, $\widehat Z_{m+1}:=2qH_{m+1}-I$, and $u_m^{\rm SI}:=V_{m+1}g_q(\widehat Z_{m+1})e_1\|b\|_2$, where $e_1$ is the first coordinate vector in $\C^{m+1}$.  This is the transformed-variable RD--Arnoldi projection of \citet{MoretNovati2004}.  The companion half-line paper \citep{XueZhang2026} instead compresses the original matrix; the search space is the same, but the projected vector generally is not.

\begin{corollary}[Projected shift-and-invert Arnoldi bound]\label{cor:projected-shift-invert}
Assume $W(A)\subset\Sigma_\beta$.  For every $q>0$ and $m\ge0$,
\begin{equation}\label{eq:projected-SI-bound}
  \|\e^{-hA}b-u_m^{\rm SI}\|_2
  \le 2(1+\sqrt2)\,E_{m,\beta}(q)\,\|b\|_2.
\end{equation}
Consequently, for $m\ge1$ and $q_m=\alpha m$ with $\alpha$ in a compact $K\Subset(0,\alpha_c(\gamma))$, there exists $C_{K,\beta}>0$, independent of $m$, $A$, and $b$, such that
\[
  \|\e^{-hA}b-u_m^{\rm SI}\|_2
  \le C_{K,\beta}\,\|b\|_2\,m^{-1/2}H_\beta(\alpha)^m.
\]
\end{corollary}

\begin{proof}
We first verify the numerical-range inclusion $W(Z_q)\subset K_\beta$ directly. Let $x\in\C^N$ with $\|x\|_2=1$, set $y=B_qx$, and write $u:=q x^*B_qx=q x^*y$. Since $x=(qI+hA)y$, Cauchy--Schwarz and $W(A)\subset\Sigma_\beta$ give
\[
  \Re u=q^2\|y\|_2^2+qh\Re(y^*Ay),
  \qquad
  |u|^2\le q^2\|y\|_2^2,
  \qquad
  |\Im u|=qh|\Im(y^*Ay)|.
\]
Put $\theta=\beta/2$. If $\theta=0$, then $u$ is real and $0<u\le1$, so $x^*Z_qx=2u-1\in(-1,1]$. If $0<\theta<\pi/2$, the sector condition yields
\[
  |\Im u|
  \le \tan\theta\,qh\Re(y^*Ay)
  \le \tan\theta\,(\Re u-|u|^2).
\]
For $z:=x^*Z_qx=2u-1$, this is equivalent to
\[
  1-|z|^2\ge 2|\Im z|\cot\theta,
\]
which, by the disk representation~\eqref{eq:lens-disks}, is exactly $z\in K_\beta$. Hence $W(Z_q)\subset K_\beta$. Moreover $\Re u\ge q^2\|B_qx\|_2^2>0$; compactness of the unit sphere therefore shows that $W(Z_q)$ stays a positive distance from the boundary point $-1$, where $g_q$ has its essential singularity. Compression gives $W(\widehat Z_{m+1})\subset W(Z_q)$.  For any $p\in\Pm_m$, Arnoldi polynomial exactness yields $p(Z_q)b=V_{m+1}p(\widehat Z_{m+1})e_1\,\|b\|_2$. Choose a best polynomial $p_m^*$ for $g_q$ on $K_\beta$.  Subtracting this identity from the two functional-calculus terms gives
\[
  \e^{-hA}b-u_m^{\rm SI}
  =(g_q-p_m^*)(Z_q)b
  -V_{m+1}(g_q-p_m^*)(\widehat Z_{m+1})e_1\,\|b\|_2.
\]
Apply the Crouzeix--Palencia inequality to $Z_q$ and to $\widehat Z_{m+1}$ and use the two numerical-range inclusions.  Each term is at most $(1+\sqrt2)E_{m,\beta}(q)\|b\|_2$, which proves \eqref{eq:projected-SI-bound}; Theorem~\ref{thm:best-error} gives the final estimate.
\end{proof}

The reuse of one factorization or preconditioner refers to a run with a prescribed degree budget $m$ and its corresponding pole parameter $q_m$ (physical repeated pole $s_m=-q_m/h$).  If the budget is changed and $q_m=\alpha m$ is retuned, the shifted matrix changes as well, and a new run is necessary. 

Finally, the unique minimizer of the proved noncoalescing rate supplies a sector-informed repeated shift.  With
\begin{equation}\label{eq:matrix-alpha-star}
  q_{m,\beta}^*:=\alpha_\beta^*m,
  \qquad
  s_{m,\beta}^*:=-\alpha_\beta^*m/h,
\end{equation}
the upper bounds in Corollaries~\ref{cor:shift-invert-bound} and~\ref{cor:projected-shift-invert} have root factor
\begin{equation}\label{eq:matrix-rho-star}
  \rho_\beta^*
  =\tan\!\left(\frac{\pi}{4\gamma_\beta}\right).
\end{equation}
For $\beta=0$ this is the classical global half-line root-rate optimum.  For $\beta>0$, \eqref{eq:matrix-alpha-star} minimizes only the proved noncoalescing sector-uniform rate, not a matrix-specific or unproved objective.  The simple choice $q_m=m$ remains a nearby angle-independent alternative.

\subsection{A nonnormal Arnoldi experiment}\label{subsec:matrix-experiment}

We test the projected approximation $u_m^{\rm SI}$ on a family whose numerical range is controlled directly. For $N=64,128$, let $D_N$ be diagonal with
\[
 (D_N)_{j+1,j+1}=10^{4((37j)\bmod N)/(N-1)},\quad 0\le j<N,
 \qquad
 A_N=D_N+\sqrt3\,D_N^{1/2}S_ND_N^{1/2},
\]
where $S_N$ has $1/2$ on the first upper diagonal, $-1/2$ on the first lower diagonal, and zeros elsewhere. Since $S_N^*=-S_N$ and $\|S_N\|_2\le1$, every complex vector $v$ satisfies
\[
 |\Im(v^*A_Nv)|\le\sqrt3\,v^*D_Nv
 =\sqrt3\,\Re(v^*A_Nv).
\]
Thus $W(A_N)\subset\Sigma_{2\pi/3}$. The diagonal and skew-symmetric parts do not commute, so $A_N$ is nonnormal. We take $h=0.1$ and normalize the vector with entries $\sin(\sqrt2 j)+\cos(\sqrt3 j)$, $1\le j\le N$, to obtain $b$.

For each degree $m$ and each rule $q=\alpha_{2\pi/3}^*m$ or $q=m$, we perform a separate Arnoldi run on $qB_q$ with two orthogonalization passes. This scalar rescaling gives the same projected approximation as $B_q$. Each run reuses one LU factorization and uses $m+1$ shifted solves, including the last product needed to assemble the $(m+1)$-dimensional compression. A dense scaling-and-squaring exponential supplies the reference action; an independent exponential-action calculation differs by less than $4.5\times10^{-15}\|b\|_2$. A third orthogonalization pass changes the computed actions by less than $2.1\times10^{-14}\|b\|_2$. We display only $m\le28$; the $m=32$ errors are too close to these checks to interpret as convergence data.

\begin{figure}[htbp]
\centering
\includegraphics[width=.98\textwidth]{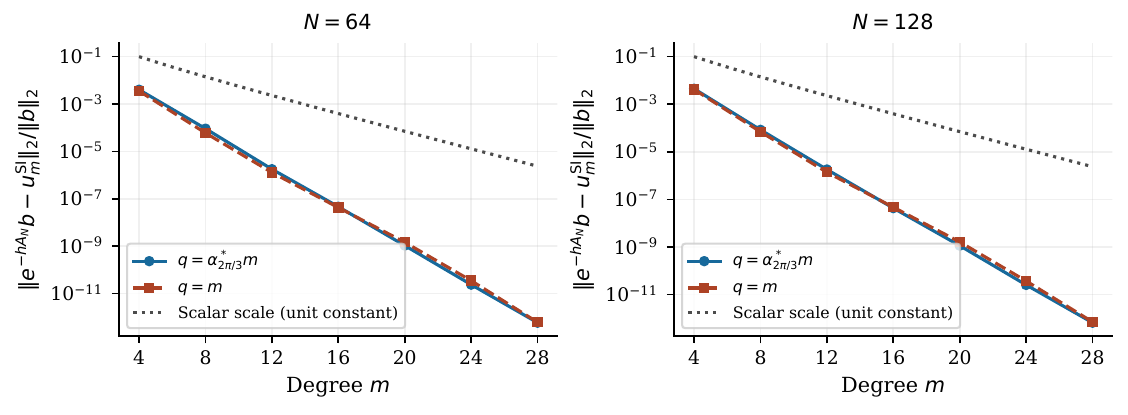}
\caption{Actual projected Arnoldi errors $\|\e^{-hA_N}b-u_m^{\rm SI}\|_2/\|b\|_2$ for two nonnormal matrices. Each marker is a separate fixed-shift run. The dotted curve is the scalar scale $m^{-1/2}(\rho_{2\pi/3}^*)^m$ with unit constant, shown only for comparison; it is not a numerical error bound.}
\label{fig:arnoldi}
\end{figure}

The two dimensions give similar behavior, and both matrices converge faster than the sector-uniform scalar scale. The optimized ratio is not uniformly better at small degrees. This is consistent with Corollary~\ref{cor:projected-shift-invert}: it gives a dimension-independent upper bound, while a particular matrix and vector can exploit much more information than containment in an infinite sector.

\section{Concluding remarks}\label{sec:conclusion}

For each fixed sector opening, the boundary Faber integral and its global contour deformation give a two-saddle expansion. A two-coefficient noncancellation argument and a uniform tail estimate turn that expansion into $E_{m,\beta}(\alpha m)\asymp m^{-1/2}H_\beta(\alpha)^m$ on compact subcritical ratio sets. On the proved noncoalescing interval, the exponential factor has an explicit angle-dependent minimizer. The essential addition to the companion half-line analysis is the fractional-power contour geometry needed for this sharp two-sided sector order.

The scalar result gives sector-uniform matrix bounds and an upper bound of the same order for projected shift-and-invert Arnoldi. It does not give a matching lower bound for an individual Arnoldi run or a matrix-specific optimal pole. The finite-degree pole search and the nonnormal experiment illustrate these distinctions, with the simple rule \(q=m\) remaining competitive at the tested degrees. The theory also provides a sharp benchmark for this rule, whose asymptotic degree overhead relative to the subcritical rate minimizer is quantified by \eqref{eq:degree-ratio}.

The analysis does not cover coalescence, vanishing ratios, or a joint limit as $\beta\uparrow\pi$. Coalescence suggests an Airy-type local model, while vanishing ratios require a sectorial counterpart of the companion's fixed-to-moving-pole transition. Another open question is which computable interpolation schemes inherit the sharp $m^{-1/2}$ scale: subexponential stability alone preserves only the root factor.

\section*{Acknowledgements}

\noindent\textbf{Funding.} The first author was partially supported by NSF grant DMS-2608905.

\smallskip

\noindent\textbf{Conflict of interest.} The authors declare no conflicts of interest.

\smallskip

\noindent\textbf{Data and code availability.} Supporting computational scripts, raw arrays, and auxiliary validation data are available from the corresponding author upon reasonable request.

\smallskip

\noindent\textbf{Use of artificial intelligence tools.} OpenAI ChatGPT (GPT-5.6 Sol and GPT-6 Astra) was used to suggest candidate derivation routes and assist with consistency checks for parts of the saddle analysis and error-transfer arguments, as well as literature comparison, exposition, and code review. The authors carefully reviewed and checked all AI-assisted materials and assume full responsibility for the mathematical correctness, citations, computational results, scholarly content, and conclusions.

\end{document}

%% file: pole_selection_table.tex
\begin{table}[tbp]
\centering
\caption{Finite-degree pole selection on $[0.35,1.35]$. Here $\widetilde E_m^*=\widetilde E_m(\widetilde\alpha_m)$; all error columns are numerical diagnostics.}
\label{tab:finite-poles}
\begin{tabular}{c r c c c c c}
\hline
$\beta/\pi$ & $m$ & $\widetilde\alpha_m$ & $\sqrt m|\widetilde\alpha_m-\alpha_\beta^*|$ & $\widetilde E_m^*$ & $\widetilde E_m(\alpha_\beta^*)/\widetilde E_m^*$ & $\widetilde E_m(1)/\widetilde E_m^*$ \\
\hline
$1/3$ & 8 & $1.085$ & $0.970$ & $2.975\times10^{-4}$ & $1.527$ & $1.498$ \\
$1/3$ & 12 & $0.702$ & $0.136$ & $1.626\times10^{-5}$ & $1.380$ & $1.196$ \\
$1/3$ & 16 & $0.723$ & $0.076$ & $9.404\times10^{-7}$ & $1.239$ & $2.069$ \\
$1/3$ & 20 & $0.735$ & $0.030$ & $5.643\times10^{-8}$ & $1.100$ & $2.551$ \\
$2/3$ & 8 & $1.067$ & $0.722$ & $4.961\times10^{-3}$ & $1.529$ & $1.239$ \\
$2/3$ & 12 & $0.970$ & $0.550$ & $8.279\times10^{-4}$ & $1.529$ & $1.157$ \\
$2/3$ & 16 & $0.923$ & $0.445$ & $1.452\times10^{-4}$ & $1.532$ & $1.546$ \\
$2/3$ & 20 & $0.895$ & $0.373$ & $2.616\times10^{-5}$ & $1.533$ & $1.645$ \\
\hline
\end{tabular}
\end{table}